\documentclass[12pt]{amsart}

\usepackage{amsmath,amsthm,amscd,amsfonts,wasysym,amssymb,epic,eepic,bbm,tikz-cd,mathtools,multicol,enumitem,mathrsfs}
\usepackage[pagebackref,colorlinks=true,linkcolor=blue,citecolor=blue]{hyperref}
\usepackage[noabbrev]{cleveref}
\usepackage{MnSymbol}
\makeatletter
\newcommand\myitem[1][]{\item[#1]\refstepcounter{dummy}\def\@currentlabel{#1}}
\makeatother

\allowdisplaybreaks
\newtheorem{thm}{Theorem}[section]
\newtheorem{cor}[thm]{Corollary}

\newtheorem{prop}[thm]{Proposition}
\newtheorem{remark}[thm]{Remark}

\newcounter{remarkscounter}

\numberwithin{equation}{section}
\newcommand{\A}{\mathbb{A}}
\newcommand{\GL}{\mathrm{GL}}

\newcommand{\Gal}{\mathrm{Gal}}

\makeatletter
\newcommand{\rprod}{%
  \DOTSB             
  \mathop{\mathpalette\rprod@\relax}\slimits@}
\newcommand{\rprod@}[2]{%
  \ooalign{$\m@th#1\prod$\cr$\m@th#1\coprod$\cr}}
\makeatother

\newcommand{\quash}[1]{}

\theoremstyle{definition}

\numberwithin{equation}{subsection}

\newcommand{\op}{\mathrm{op}}

\def\L{{\mathcal L}}
\def\G{{\mathbb G}}

\def\N{{\mathbb N}}

\def\Aut{{\mathrm {Aut}\,}}
\def\End{{\mathrm {End}}}
\def\Gal{{\mathrm {Gal}}}
\def\Id{{\mathrm {Id}}}
\def\ker{{\mathrm {ker}\,}}
\def\Ker{{\mathrm {Ker}\,}}

\def\im{{\mathrm {im}\,}}

\def\gp{{\mathrm {gp}}}

\def\Spec{{\mathrm {Spec}}}

\def\dim{{\mathrm{dim}}}

\def\GL{{\mathrm{GL}}}

\def\C{{\mathbb{C}}}
\def\F{{\mathbb{F}}}

\def\A{{\mathbb{A}}}
\def\Z{{\mathbb{Z}}}
\def\Q{{\mathbb{Q}}}
\def\P{{\mathbb{P}}}

\def\Res{{\mathrm{Res}}}

\def\Ext{{\mathrm{Ext\,}}}

\def\Pic{{\mathrm{Pic}}}
\def\Set{{\mathrm{Set}}}

\def\Div{{\mathrm{Div}}}
\def\Frob{{\mathrm{Frob}}}

\def\op{{\mathrm{op}}}

\def\ab{{\mathrm{ab}}}

\allowdisplaybreaks

\begin{document}
\linespread{1.2}

\title[Log Geometry \& GCF]{Logarithmic Geometry and Geometric Class Field Theory}

\author{Aaron Slipper}
\address{Department of Mathematics\\
Duke University\\
Durham, NC 27708}
\email{aaron.slipper@duke.edu}

\begin{abstract}
    We demonstrate an application of logarithmic geometry in the context of geometric Langlands, by providing a logarithmic upgrade of Deligne's geometric class field theory for tamely ramified Galois groups. In particular, we define a framed logarithmic Picard space, and show that a logarithmic compactification of the classical tamely ramified Div-to-Pic map has, for sufficiently large degree, log simply connected fibers given by logarithmically compactified vector spaces. This provides a canonical bijection between local systems on the curve with divisorial log structure and multiplicative local systems on the framed logarithmic Picard (a logarithmic version of the Hecke eigensheaf correspondence of geometric Langlands for $\GL_1$). We use this to re-derive tamely ramified global Artin reciprocity for function fields, and show that logarithmic geometry allows for a geometric interpretation of local-to-global compatibility at \textit{all} places (in addition to the unramified places). 
\end{abstract}

\maketitle

\setcounter{tocdepth}{2}
\tableofcontents

\section{Introduction}\label{Intro}

\subsection{Motivation and History} Let $X$ be a smooth, projective curve over a finite field $k$. The centerpiece of geometric class field theory is a canonical equivalence between the category of 1-dimensional $\ell$-adic local systems on $X$ and 1-dimensional multiplicative $\ell$-adic local systems on $\Pic(X)$. This in turn provides a canonical equivalence between 1-dimensional representations of the (abelianized) \'etale fundamental group of $X$ and the 1-dimensional $\ell$-adic representations of $\Pic_X(k)$. From this one can deduce (unramified) global Artin reciprocity for the function field $K := K(X)$.

Deligne provided an elegant geometric argument for the first of the above equivalences (see \cite{toth2011geometric} for an excellent discussion). It can be summarized as follows. Let $X^{(n)}$ denote the $n$th symmetric power of $X$, whose functor of points represents divisors of degree $n$ on $X$. We observe that, for sufficiently large $n$, the map $X^{(n)} \to \Pic^{n}_X$ which sends a divisor to the corresponding invertible sheaf, is a $\P^N$ fibration (with $N$ an integer depending on $n$ and the genus of $X$). Since $\P^N$ is simply connected, we find that the \'etale fundamental groups of $X^{(n)}$ and $\Pic^n_X$ are isomorphic for $n\gg 0$. 

We may apply this observation as follows. Given a 1-dimensional local system $\mathfrak{L}$ on $X$, one can canonically extend it to a multiplicative local system on $X^{(n)}$ for all $n$: take the $n$-fold product $\boxtimes^n \mathfrak{L}$ on $X^n$, push forward along the finite map $p_n: X^n \to X^{(n)}$, and take $S_n$-invariants. The isomorphism of $\pi_1^{et}\,X^{(n)}$ with $\pi_1^{et}\, \Pic^n_X$ 
for sufficiently large $n$ implies that for $n \gg 0$ we may uniquely transport $\mathfrak{L}^{(n)}: = \left(p_{n*}\boxtimes^n \mathfrak{L}\right)^{S_n}$ to a 
local system on $\Pic^n_X$. We can then utilize the multiplicative structure of $\Pic(X)$ to (uniquely) extend this family of local systems $\mathfrak{L}^{(n)}$ to a multiplicative local system on \textit{all} of $\Pic = \bigsqcup_{n \in \Z} \Pic^n$. 

On the other hand, given a 1-dimensional multiplicative local system on $\Pic_X$, we may pull it back to $X$ along the Abel-Jacobi map $X \to \Pic^1_X: \, x \mapsto O(x)$. This establishes the canonical correspondence between 1-dimensional local systems on $X$ and 1-dimensional multiplicative local systems on $\Pic_X$, which may be viewed as the ``Hecke eigensheaf" correspondence in geometric Langlands for $\GL_1$ \cite{BeilinsonDrinfeldHecke}.

Deligne's correspondence can be used as the critical ingredient in proving global class field theory for the function field $K$. In particular, it permits one to show that the Artin map, sending $x \in \Div_X(\F_q)$ to the (arithmetic) Frobenius $\Frob_x \in \pi_1^{\mathrm{ab}}(X) \cong \Gal(K^{\ab, \mathrm{ur}}/K)$, factors through $\Pic_X(\F_q)$. Interpreting $\Pic_X(\F_q)$ (via the Weil uniformization) as an id\`elic double quotient, this is equivalent to showing that the reciprocity map $\mathbb{I}_K/\mathbb{O}_K \to \Gal(K^{\ab, \textrm{ur}}/K)$ is \textit{automorphic} -- i.e., factors through global points $K^{\times} \hookrightarrow \mathbb{I}_K$. From this we may recover the usual formulation of unramified global class field theory for $K$.

The tamely ramified case of geometric class field theory can be treated similarly, although the argument is slightly more complicated. Here we pick a reduced closed subscheme $S \subset X$ (i.e., a finite union of points) which we declare to be our branching locus, and we consider 1-dimensional representations of the tamely ramified \'etale fundamental group over $X(S)$.

The analogue of $\Pic$, called $\Pic_{X,S}$ (with $S \subset X(k)$ the branching locus), represents (on $k$-points) invertible sheaves $\mathcal{L}$ on $X$ with a fixed rigidification $\mathcal{L}|_s \cong \A^1$ for each closed point $s \in S$. This is a semi-abelian variety (in particular, it is a commutative group scheme but not necessarily proper). 

Letting $U \subset X$ denote the open complement of $S$, one considers the map $U^{(n)} \to \Pic_{X,S}^n$ from the symmetric powers of $U$ to $\Pic_{X,S}$. But this map is not proper; one must construct a relative compactification $U^{(n)} \subset \overline{U^{(n)}} \to \Pic_{X,S}$ where $\overline{U^{(n)}} \to \Pic_{X,S}$ is a proper $\P^N$ fibration. One must then show that a 1-dimensional local system on $U^{(n)}$ extends (under direct image) to local systems on $\overline{U^{(n)}}$ \cite{toth2011geometric}. This then allows one to repeat Deligne's simply connectedness argument proving a canonical equivalence between multiplicative 1-dimensional local systems on $\Pic_{X,S}$ and tamely ramified 1-dimensional local systems on $U$.

Having accomplished this, one finds that for $k = \F_q$, the reciprocity map sending $x \in U(\F_q)$ to $\Frob_x$ factors through $\Pic_{X,S}(\F_q)$. As in the unramified setting, this leads one to a proof of global tamely ramified class field theory.

A raison d'\^{e}tre of logarithmic geometry is to provide a canonical means of compactifying moduli spaces, while preserving relevant structures on the boundary locus. In the above setting, we ``upgrade" the usual $X^{(n)} \to \Pic_X^n$ to a ``log" version: 

\[
 X^{(n)}_{\log} \to \overline{J}_{X, S}^n,
\]

\noindent where $\overline{J}_{X, S}^n = \overline{\Pic}_{X,S}^{n}$ is the degree-$n$ component of the log-compactified Picard and $X^{(n)}_{\log}$ is the log-scheme associated to the divisor $S^{(n)} := \{\{x_1, \ldots, x_n\}\in X^{(n)}: x_i \in S \text{ for some }i\}$ in $X^{(n)}$. This provides a (log)-proper map, whose geometry is interesting in its own right. For $n=1$, we get a map $X(\log S) \to \overline{J}_{X,\log S}^1$, the analogue of the Abel-Jacobi map.

Moreover, we will show that for sufficiently large $n$, the above map has log-simply connected fibers (given, in fact, by log-compactified vector spaces). Repeating Deligne's argument for unramified class field theory (in particular, bypassing ad hoc compactifications), we find that there is a canonical isomorphism between local systems on $X(\log S)$ and multiplicative local systems on $\overline{\Pic}_{X,S}$. 

Local systems on $X(\log S)$ 
correspond to representations of 
$\pi_{1}^{\log}X(\log S)$. 
This log-fundamental group is in turn isomorphic to the tame (Grothendieck-Murre) fundamental group of $X$ with branching locus $S$ \cite{GrothendieckMurre1971}. The equivalence between the log fundamental group of a scheme with divisorial log structure and the tamely ramified fundamental group of Grothendieck-Murre is known \cite{IllusieLogEt, FujKat, Gabber:Ramero}, though exact references have historically been tricky to locate. We review this issue, offering a proof of this fact, in Section \ref{TameLog}.

We thus see that logarithmic geometry naturally lends itself to tamely ramified geometric class field theory in two ways: firstly, through providing a good means of compactifying the map $U^n \to \Pic_{X,S}^n$, and secondly, through $X(\log S)$ having a fundamental group that is ``automatically" the tamely ramified fundamental group. Moreover, we find that by introducing logarithmic structures, we may offer a geometric interpretation of the ramified places $\widehat{K_s}^{\times}/\widehat{\mathcal{O}_s}^{\times(1)}\hookrightarrow K^{\times}\backslash\mathbb{I}_K/\mathbb{O}_{K,S}^{\times} \to \Gal(K^{t,\ab}/K)$ from local class field theory. Thus the use of logarithmic geometry allows us to put ramified places on par with the unramified places in geometric class field theory.

We close the introduction by situating this paper among several closely related geometric approaches to ramified class field theory. Campbell--Hayash formulate a broad Albanese and Cartier-duality framework for maps from a smooth curve to commutative group stacks, and use it to reprove local and global geometric class field theory with arbitrary ramification \cite{CampbellHayash2021CartierDuality}. Guignard generalizes Deligne's argument to relative curves with arbitrary ramification, using rigidified Picard schemes and a relative form of geometric global class field theory \cite{Guignard2019RamifiedRelativeCurves}. Takeuchi gives another proof for curves over a perfect field by compactifying Abel-Jacobi maps through blowups of symmetric powers; unlike Deligne's original symmetric-power argument, this treats arbitrary bounded ramification, expressed in terms of a modulus and Swan-conductor bounds \cite{Takeuchi2019BlowUps}. The present paper is complementary to these works: we restrict to the tame, divisorial-logarithmic setting and stay close to Deligne's symmetric-power argument, but replace auxiliary compactifications by a logarithmic Picard space and a logarithmic Div-to-Pic map whose fibers are log-compactified vector spaces. In particular, the ramification divisor is built into both the curve and the Picard side as logarithmic structure; this is the point at which logarithmic geometry enters the construction, and it is also what makes the local-to-global compatibility at the ramified points visible geometrically.

\subsection{Outline of This Paper} We begin in Section \ref{LogGeo} by reviewing what we will need from the theory of fs log schemes. We discuss the functor $\G_m^{\log}$ and the notion of log properness, which is defined through valuative criteria (and first appeared in the relatively recent \cite{MolchoWise2018TheLP}). In particular, we discuss why $\G_m^{\log}$ is proper.

\vspace{3mm}

In Section \ref{Kummerpi1etal} we review the basics of the Kummer \'etale topology, leading up to a discussion and proof of:

\vspace{3mm}

\textbf{Corollary} \ref{FundIso}: 
\textit{Let $X$ be a locally Noetherian, regular scheme with normal crossings divisor $D$. Let $\widetilde{x}$ be a log geometric point over $ x \in D$, with image $\overline{x}$ as a geometric point of $D$. Then we have a canonical isomorphism }
    
    \[
    \pi_1^{\log}(X(\log D), \widetilde{x})  \cong \pi_1^t(U, \overline{x}),
    \]

    \noindent \textit{where the right hand side is the tamely ramified \'etale fundamental group of Grothendieck-Murre.}

\vspace{3mm}

\noindent We conclude this section with a discussion of the homotopy exact sequence for $\pi_1^{\log}$ in the context of fs log schemes and more general logarithmic spaces, and we discuss in particular the fundamental group of $\G_m^{\log}$.

\vspace{3mm}

In Section \ref{compvecsp}, we define the notion of a compactified vector bundle, and show that the projection map induces an isomorphism of the total space with the base (\ref{simpleconnfiber}).

\vspace{3mm}

In Section \ref{LogPicRig}, we introduce the rigidified logarithmic Picard $\overline{\Pic}_{X,S} = \sqcup_n \overline{J}_{X,S}^n$ for a smooth projective curve $X$ with fixed ramification locus $S$. We show 

\vspace{3mm}

\textbf{Proposition} \ref{propernessJac}.  \textit{For each $n$, $\overline{J}_{X,S}^n$ is log proper. In particular $\overline{J}_{X,S} =\overline{J}_{X,S}^0$ is a log-proper abelian group-valued functor on the category of fs log schemes.}

\vspace{3mm}

We then discuss the notion of character sheaves on abelian group spaces in the category of log schemes. We show (Proposition \ref{MultSysGmlog}) that character sheaves on the classical rigidified Picard uniquely extend to the log compactification $\overline{\Pic}_{X,S}$.

\vspace{3mm}

In Section \ref{GeomSymPic}, we prove our main theorem regarding the geometry of the map $X^{(n)}_{\log} \to \overline{J}_{X,S}^{n}$:

\vspace{3mm}
 \textbf{Theorem} \ref{mainthm}. \textit{Say $|S|>1$, and $n \ge 2g-1 + |S|$. The map $\psi_n: X^{(n)}_{\log} \to \overline{J}_{X,S}^{n}$ is surjective; moreover, it is valuative-locally a compactified vector bundle, with fibers of dimension $\dim H^0(X, \mathcal{L}) - |S|$ for (any) $\mathcal{L} \in \Pic_X^n(k)$.}
\vspace{3mm}

In Section \ref{ClassClassfieldtheory}, we prove

\vspace{3mm}
\textbf{Theorem} \ref{LogGCFT}. \textit{Let $\overline{x}$ be a geometric point of $X$ lying over a rational point $x \in X(k)$. There is a canonical equivalence between}
    
    \textit{1) 1-dimensional $\ell$-adic local systems $\mathfrak{L}$ on $X(\log S)$, with a fixed isomorphism $\mathfrak{L}_{\overline{x}} \simeq \overline{\mathbb{Q}_{\ell}}$; and} 
    
    \textit{2) multiplicative 1-dimensional $\ell$-adic local systems (character sheaves) on $\overline{\Pic}_{X,S}$.}
\vspace{3mm}

We then follow Deligne's argument for tamely ramified geometric class field theory, which gives us the familiar version of class field theory.

\vspace{3mm}
\textbf{Theorem} \ref{ClassClassfieldtheory}. \textit{There is a map $\textrm{Artin}: \overline{\Pic}_{X,S}(\F_q) \to \pi_1^{\textrm{log}, \textrm{ab}}(X(\log S))$, that corresponds to the tamely ramified global Artin map under the Weil uniformization}

\[
\overline{\Pic}_{X,S}(\F_q) \simeq K^{\times} \backslash \mathbb{I}_K/\mathbb{O}_{K,S} \simeq K^{\times}_S \backslash \mathbb{I}_K^{(S)}/\mathbb{O}_{K}^{(S)},
\]

\noindent \textit{and under the isomorphism $\pi_1^{\log\textrm{ab}}(X(\log S)) \simeq \Gal\left(K^{t,S,\ab}/K\right)$. This map induces isomorphisms on finite quotients, showing that abelian extensions of $K$ tamely ramified over $S$ correspond to finite-index subgroups of $\Pic_{X,S}(\F_q)$.}
\vspace{3mm}

The new point of this paper is not the statement of tame Artin reciprocity by itself, but rather the logarithmic geometric mechanism which recovers it: the framed logarithmic Picard, the log-compactified Div-to-Pic map, the compactified-vector-space description of its fibers, and the resulting logarithmic interpretation of local-to-global compatibility at ramified places. In Section \ref{locglobcomp}, we discuss how one can interpret the canonical map $\textrm{Rec}_{\textrm{loc}}: \G_m^{\log}(\F_q^{\log}) \to \pi_1^{\log \ab}(\F_q^{\log})$ as the correct analogue of \textit{local} logarithmic geometric class field theory (for $\F_q^{\log}$ the standard log point, see Section \ref{htpylogstack}). For $\log(s) = \F_q^{\log}$ a logarithmic point of the curve, we construct natural maps $i_s: \mathbb{G}_m^{\log}(\log(s)) \to \overline{\Pic}_{X,S}(\F_q)$ (\ref{imap}), and we have a natural map $s_*:\pi_1^{\log \ab}(\F_q^{\log}) \to \pi_1^{\log \ab}(X(\log S))$ (\ref{s*map}). This leads us to

\vspace{3mm}
\textbf{Theorem }\ref{loctoglo}.\textit{ (Logarithmic Local-to-global compatibility.) The following diagram commutes:}

\begin{equation*}
\begin{tikzcd}
\mathbb{G}_m^{\log}(\mathbb{F}_q^{\log}) \arrow[r, "\textrm{Rec}_{\textrm{loc}}"] \arrow[d, "i_s"'] & \pi_1^{\log \textrm{ab}}(\mathbb{F}_q^{\log}) \arrow[d, "s_*"] \\
{\overline{\textrm{Pic}}_{X,S}(\mathbb{F}_q)} \arrow[r, "\textrm{Artin}"]                           & \pi_1^{\log \textrm{ab}}(X(\log S)),                           
\end{tikzcd}
\end{equation*}

\noindent \textit{and this corresponds to the map}

\begin{equation*}
\begin{tikzcd}
\widehat{K_s}^{\times}/\widehat{\mathcal{O}_s}^{\times (1)} \arrow[r, "\textrm{Rec}_{\textrm{loc}}"] \arrow[d] & {\textrm{Gal}\left(\overline{K_s^{t,\textrm{ab}}}/K_s\right)} \arrow[d] \\
{K^{\times} \backslash\mathbb{I}_K/\mathbb{O}_{K,S}} \arrow[r, "\textrm{Artin}"]                               & {\textrm{Gal}\left(\overline{K^{t,\textrm{ab}}}/K\right)}              
\end{tikzcd}
\end{equation*}

\noindent \textit{from class field theory.}

\vspace{3mm}

\subsection{Acknowledgments} I am greatly indebted to Professor Kazuya Kato, who taught me (and the world!) logarithmic geometry and who was invaluable throughout this project -- both in the mathematics itself and in his generous mentorship. I also owe deep thanks to Patrick Kennedy-Hunt for his reading of various drafts of this manuscript and for his many helpful discussions, comments and encouragement on both this project and its potential extensions. An additional thank-you to Dmitri Whitmore for many illuminating meetings and conversations, and to Leo Herr for general inspiration as well as for pointing out some subtleties regarding saturated base change. This material is based upon work partially supported by the National Science Foundation Graduate Research Fellowship under Grant No. 2140001 and by Duke University's Number Theory RTG under DMS-2231514.

\section{Logarithmic Geometry}\label{LogGeo}

\subsection{Conventions} For us, the category $\mathbf{Log\,\, Schemes}$ will refer to saturated (though not necessarily fine) logarithmic schemes. This permits us to have the usual finite inverse and direct limits (e.g., base change), while allowing for more general test objects (e.g., valuative monoids other than $\N$). 

\subsection{The Functor \texorpdfstring{$\mathbb{G}_m^{\log}$}{Gmlog}} We briefly remind the reader of the functor $\mathbb{G}_m^{\log}: \textbf{Log Schemes}^{\op} \to \textbf{Ab}$, given by

\[
\mathbb{G}_m^{\log}(T, \mathcal{M}_T, \alpha) = \mathcal{M}_T(T)^{gp}.
\]

\vspace{3mm}

$\mathbb{G}_m^{\log}$ is not representable (see \cite{MolchoWise2018TheLP}, Proposition 2.2.7.2.). Since it interacts solely with the monoid structure of a test fs log scheme, it is an example of what Professor Kato has called a ``ghost" -- something that only touches the ``spirits" of the log world.

\subsection{Log Properness}\label{logprop} In this section we discuss the notion of ``properness" for log schemes and log spaces, which are defined in terms of a valuative criterion (\cite{MolchoWise2018TheLP}, section 2.2.5.). 

We let $D$ be a log scheme whose underlying scheme is the spectrum of a valuation ring. Recall that this means that $k[D]$ is a domain, with fraction field $K$, such that for any $f \in K$, we have at least one of $f, f^{-1} \in k[D]$.  Say we have a log structure $\mathcal{M}_\eta$ on the generic point (that is, a monoid, which we shall somewhat abusively call $\mathcal{M}_\eta$, and a map $\mathcal{M}_\eta \to K$ such that $\alpha^{-1}(K^{\times}) \to K^{\times}$ is an isomorphism). We consider the associated log structure on $D$, whose monoid $\mathcal{M}_D$ is given by the fiber product (taken in the category of monoids, where we remember only the multiplicative monoid structure of the bottom row):

\[
\begin{tikzcd}
\mathcal{M}_D \arrow[d] \arrow[r] & \mathcal{M}_\eta \arrow[d] \\
{k[D]} \arrow[r]                  & K.                         
\end{tikzcd}
\]

\noindent In other words, $\mathcal{M}_D := \alpha^{-1}(k[D])$. We shall call such a log scheme $D$ a log-annulus, since $D$ is something akin to a punctured disk with a logarithmic boundary at the puncture. Note that almost all valuative $D$ are \textit{not} fs -- they are saturated, but, other than those monoids generated by $\N$, are not fine.

Now we can define properness. Given a morphism $\phi$ between two functors $F$ and $G : \textbf{Log Scheme}^{\op} \to \textbf{Sets}$, we say that $\phi$ is \textit{(log) proper} if the diagram

\[
\begin{tikzcd}
F(D) \arrow[d] \arrow[r] & G(D) \arrow[d] \\
F(\eta ) \arrow[r]       & G(\eta)       
\end{tikzcd}
\]

\noindent is Cartesian for every log-annulus $D$ with generic point $\eta$. In particular, if $F$ and $G$ are representable and $\phi$ corresponds to a map between the corresponding logarithmic schemes, we say that this morphism of log schemes is proper. And if $X$ is a log scheme, with structure map $X \to \Spec(\Z)$ (here $\Spec(\Z)$ is given the trivial log structure), then we say that $X$ is proper if the structural morphism is proper.

We will now give two examples of log properness.

\begin{prop}\label{logproperscheme}
    If $X$ is a log scheme whose underlying scheme is proper, then $X$ is log-proper.
\end{prop}

\begin{proof}
We must verify that any map of log schemes $\eta \to X$ extends to $D \to X$. Since the underlying scheme of $X$ is proper, we may extend the underlying maps of schemes $\eta \to X$ to a map $D \to X$. Now we may work locally on a Zariski neighborhood $\Spec(A)$ of the schematic image of $D$ in $X$. We have:

\[
\begin{tikzcd}
                              & \mathcal{M}_D \arrow[dd] \arrow[ld] &                                                                          \\
\mathcal{M}_{\eta} \arrow[dd] &                                     & \mathcal{M}_X(\textrm{Spec}(A)) \arrow[dd] \arrow[ll] \arrow[lu, dashed] \\
                              & {k[D]} \arrow[ld]                   &                                                                          \\
K                             &                                     & A \arrow[ll] \arrow[lu]                                                 
\end{tikzcd}
\]

\noindent where the top layer consists of monoids and the bottom layer consists of rings, and all of the vertical arrows are the corresponding map $\alpha$. (We consider the diagram as a whole to live in the category of monoids, forgetting the ring structure of the bottom layer.) We are trying to fill in the dotted line. However, $\mathcal{M}_D$ is the fiber product $\mathcal{M}_{\eta} \times_K k[D]$ by construction, so by definition there exists a unique dotted line completing the diagram.
\end{proof}

Observe, in particular, that if $X$ is a proper scheme, then $X$ viewed as a log scheme with trivial log structure is also proper. Moreover, if $X$ is proper, so is $X(\log D)$ for a divisor $D$. This helps make precise the intuition that $X(\log D)$ is a compactification of the open set $X \setminus D$; in fact, for $X$ defined over $\C$, the associated Kato-Nakayama space \cite{katonakayama1999log} models $X (\log D)$ as a kind of real blowup of $X(\C)$ along the divisor $D$, which is indeed compact in the usual topology.

We also note that Proposition \ref{logproperscheme} readily upgrades to morphisms:

\begin{prop}
    Say $f: X \to Y$ is a map of fs log schemes, then $f$ is log proper iff it is proper viewed as a map of schemes.
\end{prop}

It turns out that our favorite nonrepresentable functor is also log proper:

\begin{prop}\label{logproperGm}
$\G_m^{\log}$ is proper.
\end{prop}

\begin{proof}
Indeed, a map $\eta \to \G_m^{\log}$ is simply an element $x \in \mathcal{M}_{\eta}^{\gp}$. We are asking if we may lift this to a map $D \to \G_m^{\log}$. Hence we are asking if $x \in \mathcal{M}_D^{\gp}$. Indeed: $(\mathcal{M}_\eta \times_{K}k[D])^{\gp} = \mathcal{M}_\eta^{\gp}$ so we have

\begin{equation}
    \mathcal{M}_D^{\gp} = \mathcal{M}_\eta^{\gp}
\end{equation}

\noindent and in particular $x \in \mathcal{M}_D^{\gp}$.
\end{proof}

\section{The Kummer \'Etale Site, Tame Ramification, and \texorpdfstring{$\pi_1^{\log}$}{pi1log}}\label{Kummerpi1etal}

\subsection{The Kummer \'Etale Site} (Cf \cite{NakLogEtI}, \cite{IllusieLogEt}). Recall that a map $f: X \to Y$ of fs log schemes is said to be log \'etale if locally we may write $f$ as a chart $(u, v,\, \varphi: Q \to P)$ (see \cite{IllusieLogEt}, 1.3), such that

i) kernel and cokernel of $u^{\gp}$ are groups of finite order, with orders invertible on $X$; and 

ii) the map of underlying schemes 

\begin{equation}\label{etsm}
X \to Y \times_{\Spec \, \Z[Q]} \Spec\, \Z[P]
\end{equation}

\noindent is \'etale. (Similarly a map $f: X \to Y$ is log smooth if the kernel and torsion component of the cokernel of $u^{\gp}$ are finite and invertible on $X$, and (\ref{etsm}) is smooth.) As shown by Kato \cite{KatoLogI}, this definition is equivalent to the natural logarithmic analogue of Grothendieck's infinitesimal lifting definition for \'etale (and smooth) morphisms of schemes. And, as we might hope, when $X$ and $Y$ are schemes (i.e., fs log schemes with trivial log structure), then $f$ is log \'etale iff the underlying map of schemes is \'etale.

We say that a map between monoids $h: M \to M'$ is \textit{Kummer} if for all $m' \in M'$, there exists an $n \in \N$ such that $(m')^n \in h(M)$. Observe that this notion really depends on the monoid structure of $h$; it does not merely depend on $h^{\gp}$ as in the definition of log \'etale. A map $X \to Y$ of fs log schemes is called \textit{Kummer} if the induced map on stalks $\mathcal{M}_{Y, f(\overline{x})} \to \mathcal{M}_{X, \overline{x}}$ is Kummer for all geometric points $\overline{x} \in X$. If in addition $f$ is \'etale, we say it is \textit{Kummer \'etale}.

Now we recall the Kummer log \'etale site. For $X$ an fs log scheme, the Kummer \'etale site, $X_{\text{ket}}$, is the category of schemes over $X$ that are Kummer \'etale over $X$. The Kummer \'etale topology is the topology on $X_{\text{ket}}$ generated by covers $u_i: T_i \to T$ such that $T$ is the set-theoretic union of the images of the $u_i$. 

Observe that this is a more restrictive topology than the so-called ``full" (or ``second") log \'etale site of Kato and Nakayama \cite{NakLogII}: ``log blowups" are not permitted in $X_{\text{ket}}$. (Log blowups are quintessential examples of maps that are log \'etale, but not Kummer.) The Kummer topology suffices for our purposes; moreover, unlike the full log \'etale site, it is subcanonical \cite{KatoLogII}.

\subsection{Kummer Covers and Tame Ramification}\label{TameLog}

In this section we will prove the following \cite{IllusieLogEt}:

\begin{thm}\label{CovThm}
    Let $X$ be a log regular fs log scheme, and let $U \subset X$ be the open locus of triviality of the log structure. Then the functor

    \[
    \textup{Kcov(X)} \longrightarrow \textup{Etcov(U)}
    \]
    \[
    Z \mapsto Z \times_X U
    \]

    \noindent from the category of finite Kummer \'etale covers of $X$ to the category of classical \'etale covers of $U$ induces an equivalence between $\textup{Kcov}(X)$ and the full subcategory of $\textup{Etcov}(U)$ consisting of covers $V \to U$ which are tamely ramified along $X - U$; that is, such that if $Z$ is the normalization of $X$ in $V$, at all points with $\dim\, \mathcal{O}_{X,x}= 1$, the restriction of $Z$ to $\Spec \, \mathcal{O}_{X,x}$ is tamely ramified. 
\end{thm}

This statement appears without proof in Illusie's article  \cite{IllusieLogEt}. Loc. cit. refers to a paper of K. Kato and K. Fujiwara \cite{FujKat} which was, the author has learned from private communication, likely to have been lost. While the author was working on the present project, however, the proof finally appeared in full print in \cite{Gabber:Ramero}, Theorem 13.3.45. We discuss the proof presently.

Theorem \ref{CovThm} is a fairly straightforward consequence of Abhyankar's Lemma. Let $X=\Spec(A)$ be a regular local scheme, $D = \sum_{1\le i \le r} \textrm{div} \, f_i$ a divisor with normal crossings, where the $f_i$ are elements of the maximal ideal of $A$ belonging to a regular system of parameters. We now quote from SGAI, Appendice I \cite{SGAI} (for a proof, see loc. cit):

\begin{prop} \textbf{``The (Absolute) Abhyankar Lemma."}
Let $X$ be a regular local scheme, 
\[
D = \sum_{i=1}^{r} \operatorname{div} f_i
\]
a normal crossings divisor [...], let $Y = \operatorname{Supp} D$, and set $U = X - Y$. Let $V$ be an étale covering of $U$ which is tamely ramified with respect to $D$. For each index $i$, if $x_i$ denotes the generic point of the closed subscheme $V(f_i)$, then $\mathcal{O}_{X,x_i}$ is a discrete valuation ring with fraction field $K_i$, and one has
\[
V\vert_{K_i} = \operatorname{Spec}\Biggl(\prod_{j\in J_i} L_j\Biggr),
\]
where the $L_j$ are finite separable extensions of $K_i$. Denote by $n_j$ the order of the inertia group of a Galois extension generated by $L_j$, and let $n_i$ be the least common multiple of the $n_j$ as $j$ runs over $J_i$. If we set
\begin{equation}\label{KumCov}
X' = X[T_1,\dots,T_r]\,/\,\bigl(T_1^{n_1}-f_1,\dots,T_r^{n_r}-f_r\bigr),
\end{equation}
and
\[
U' = U_{(X')},\quad V' = V_{(X')},\quad \text{etc.},
\]
then the étale covering $V'$ of $U'$ extends uniquely (up to unique isomorphism) to an étale covering of $X'$, and the integers $n_i$ are prime to the residue characteristic $p$ of $X$.
\end{prop}

Let $n_1, \ldots n_r$ be a sequence of natural numbers. Assume that $k$ contains all of the $\textrm{LCM}(n_1, \ldots, n_r)$th roots of unity. For an affine $X$, with $f_1, \ldots f_r$ a regular sequence in $k[X]$, let us call a map $X' \to X$ of the form $X' = X[T_1,\dots,T_r]\,/\,\bigl(T_1^{n_1}-f_1,\dots,T_r^{n_r}-f_r\bigr)$, as in  \ref{KumCov}, a \textit{tamely ramified special Kummer cover} of $X$. Restricting this to the \'etale locus $V' \to U$, we shall call such a map simply a \textit{special Kummer cover} of $U$. (Note that this notion is purely scheme-theoretic; it does not involve any log structures.) An immediate consequence of Abhyankar's Lemma is that any \'etale cover $V$ of $U$, such that the normalization of $X$ in $V$ is tamely ramified over $X \setminus U$, can be written as a quotient (under the action of a subgroup $\Gamma$ of the full covering group $\mu_{n_1} \times \cdots \times \mu_{n_r}$) of a special Kummer cover:

\begin{equation}\label{CovDiag}
\begin{tikzcd}
X' \arrow[d] & U' \arrow[l, hook'] \arrow[d, "/\Gamma"] \arrow[d] \arrow[dd, "/\mu_{n_1} \times \cdots \times \mu_{n_r}", bend left, shift left=2] \\
Z \arrow[d]  & V \arrow[d] \arrow[l, hook']                                                                                                        \\
X            & U, \arrow[l, hook']                                                                                                                 
\end{tikzcd}
\end{equation}

\noindent i.e., $V = U'/\Gamma$. Moreover, for $X$ affine, it follows from Galois theory that if $Z$ is the normalization of $V$ in $X$, then $Z = X'/\!/\Gamma$, where the quotient here is the categorical quotient, $X'/\!/\Gamma := \Spec k[X']^{\Gamma}$.

We may now prove Theorem \ref{CovThm}.

\begin{proof} (Of Theorem \ref{CovThm}). The map

 \begin{align}\label{restrmor}
    \textup{Kcov(X)} &\rightarrow \textup{Etcov(U)}\\
    Z &\mapsto Z \times_X U \nonumber
\end{align}

\noindent is stated in the theorem. We now construct the functor in the other direction. Assume without loss of generality that $X$ is affine, and base changing if necessary, that $k$ contains all prime-to-$\textrm{char}(k)$ roots of unity. Consider the divisorial log structure on $X$ with locus of triviality $U$, generated by $f_1, \ldots ,f_r$. (That is to say, $U = \Spec \, k[ X]_{(f_1 \cdots f_r)}$.) Say $V \to U$ is an \'etale covering of $U$ whose normalization $Z$ in $X$ is tamely ramified over the divisor $X \setminus U$. Now we may apply Abhyankar's Lemma, which shows that $V$ is the quotient of a special Kummer cover $U'$ of $U$ by an action of a subgroup $\Gamma \subseteq \mu_{n_1} \times \cdots \times \mu_{n_r}$, as in (\ref{CovDiag}). We have $Z = X'/\!/\Gamma = \Spec\, k[X']^{\Gamma}$). 

We now may enhance $X'$ with the divisorial log structure generated by $T_1 \cdots T_r$; that is, with the log structure associated to the constant monoid $\N^r$ generated by the $T_i$. Sending $f_i \in \mathcal{M}_X$ to $T_i^{n_i} \in \mathcal{M}_{X'}$ defines a Kummer log-\'etale cover of $X$ by $X'$, and we may descend this log-structure under the action of $\Gamma$ to $Z$: consider the multiplicative monoid of $\Gamma$-invariant monomials $T_i$. By $\Gamma$-invariance, these will descend to regular functions on $Z$, which we may use to generate a monoidal log structure on $Z$. This upgrades $Z$ to a Kummer log \'etale cover of $X$. Indeed, it is precisely the log structure on $Z$ that is strictly pulled back from that of $X$.

Checking that this construction is fully faithful and an inverse to \ref{restrmor} is straightforward. 

\end{proof}

\subsection{Kummer \texorpdfstring{$\pi_1^{\log}$}{pi1log}}  We now recall that a log geometric point is a logarithmic scheme $(\overline{s}, M)$ where $\overline{s} = \Spec\, k$ for $k$ a separably closed field, and where $M$ is a monoid such that, for all $n \ge 1$ relatively prime to $\text{char} \,k$, multiplication by $n$ is a bijection on $\overline{M} := M/k^*$. Observe that these are reduced, irreducible log schemes whose underlying scheme is a point, over which there exist no Kummer covers. However, in general, log geometric points are not fine. A logarithmic point of a log scheme $X$ is a map $\overline{s} \to X$ from a log geometric point $\overline{s}$ to $X$; we will somewhat abusively call this morphism $\overline{s}$ as well. 

A natural way to construct a geometric point of $X$ is to take a geometric point of the underlying scheme of $X$ and consider the strict pullback of the log structure to this point; then one considers the inverse limit over all Kummer covers of this geometric-point-with-log-structure.

Given a logarithmic point $\overline{s}$ of a scheme $X$, we say that a Kummer \'etale open $U \to X$ is a Kummer neighborhood of $\overline{s}$ if $\overline{s} \to X$ factors through $U$, in which case we will (somewhat abusively) write $\overline{s} \in U$.  We now define the fiber functor:

\begin{align*}
    \text{Fib}_{\overline{s}}: \text{KCov}_X &\to \Set\\
    \widetilde{X} &\mapsto \pi_0\left(\lim_{\substack {\longleftarrow \\ U\ni \overline{s}}} \widetilde{X} \times_X U
    \right).
\end{align*}

\noindent (The reason we write the limit on the right-hand side as opposed to the cleaner fiber product $\widetilde{X} \times_X \overline{s}$ is because $\overline{s}$ is in general not fine, so this fiber product would leave the category of fs log schemes.) We may then define $\pi_1^{\log}(X, \overline{s})$ as the automorphism group of $\textrm{Fib}_{\overline{s}}$.

Theorem \ref{CovThm} therefore has the following corollary \cite{IllusieLogEt}:

\begin{cor}\label{FundIso}
    Let $X$ be a locally Noetherian, regular scheme with normal crossings divisor $D$. Let $\widetilde{x}$ be a log geometric point over $ x \in D$, with image $\overline{x}$ as a geometric point of $D$. Then we have a canonical isomorphism 
    
    \[
    \pi_1^{\log}(X(\log D), \widetilde{x})  \cong \pi_1^t(U, \overline{x}),
    \]

    \noindent where the right hand side is the tamely ramified \'etale fundamental group of Grothendieck-Murre.
\end{cor}

\begin{proof}
    Take the automorphism group of the fiber functors (with respect to $\overline{x}$) of the (equivalent) categories in Theorem \ref{CovThm}.
\end{proof}

This corollary will provide us with the bridge between tamely ramified global class field theory and logarithmic geometry.

\subsection{The Homotopy Exact Sequence for \texorpdfstring{$\pi_1^{\log}$}{pi1log}} As in Deligne's argument for geometric class field theory, we will need to use a homotopy exact sequence for the $\pi_1$'s of a fibration. The logarithmic version of the fundamental exact sequence is due to Y. Hoshi (\cite{Hoshi2009LogHomotopy}, Theorem 2). The proof directly parallels the \'etale case; it follows from a logarithmic analogue of Stein factorization. We will adopt a slightly different notation from loc. cit. (where a related but slightly different notion of ``log geometric point" is used):

\begin{thm}\label{loghtpyseq}
    Let $S$ be a log regular fs log scheme, $X$ a connected fs log scheme, and $f: X \to S$ a log smooth morphism whose underlying map of schemes is proper. Assume, moreover, that $f_*(\mathcal{O}_X) = \mathcal{O}_S$. Let $\overline{x}$ be a geometric point of $X$, and $\overline{s}:= f(\overline{x})$ be its image in $S$. Then we have an exact sequence:

    \begin{equation}\label{htpyexact}
        \lim_{\substack{\longleftarrow \\ U \ni \overline{s}}}\pi_1^{\log}(X\times_{S}U, \overline{x}) \to \pi_1^{\log}(X, \overline{x}) \to \pi_1^{\log}(S, \overline{s}) \to 1.
    \end{equation}
\end{thm}

\subsection{\'Etale \texorpdfstring{$\pi_1^{\log}$}{pi1log} of logarithmic spaces and stacks}\label{htpylogstack} We will need to upgrade Theorem \ref{loghtpyseq} from fs log schemes to more general logarithmic spaces like $\G_m^{\log}$. This upgrade is entirely formal. We will use the following: if we have a representable, log-smooth map of stacks $f: \mathfrak{X} \to \mathfrak{Y}$, with geometric points $\overline{x} \in \mathfrak{X}$ and $\overline{y} : = f(\overline{x}) \in \mathfrak{Y}$, such that for any $T$-point of $\mathfrak{Y}$ ($T$ an fs log scheme), the pullback $\mathfrak{X} \times_\mathfrak{Y} T \to T$ satisfies the conditions of Theorem \ref{loghtpyseq}, we have a corresponding homotopy exact sequence:

 \begin{equation}\label{htpyexactstack}
        \lim_{\substack{\longleftarrow \\ U \ni \overline{y}}}\pi_1^{\log}(\mathfrak{X}\times_{\mathfrak{Y}}U, \overline{x}) \to \pi_1^{\log}(\mathfrak{X}, \overline{x}) \to \pi_1^{\log}(\mathfrak{Y}, \overline{y}) \to 1.
    \end{equation}

We will now compute $\pi_1^{\log}$ of the log space $\G_m^{\log}$. We recall that $\G_m^{\log}$ is covered by $\P^1$ with divisorial log structure at 0 and $\infty$ (an fs log scheme, which we shall denote by $\P^1_{\log}$). Explicitly the map is given as follows: let $\P^1 = \Spec\, k[x] \sqcup_{\Spec\, k[x, x^{-1}]} \Spec\, k[x^{-1}]$, with $\A^+$ and $\A^-$ the two open sets of the cover, with $x$ and $x^{-1}$ generating the log structures at 0 and infinity. (We are being slightly abusive, writing $x$ as a local section of both $\mathcal{O}_{\P^1}$ and $\mathcal{M}_{\P^1_{\log}}$.) Given a $T$-point $f: T \to \P^1_{\log}$, then we have two open sets $f^{-1}(\A^+)$ and $f^{-1}(\A^{-})$ in $T$, at least one of which is nonempty. Thus we have $f^{-1}(x)$ and $f^{-1}(x^{-1})$; at least one of these is nonzero. We define $u \in M_T^{\gp}$ as $f^{-1}(x)$ if $\A^+$ nontrivially intersects $\im(f)$ or $f(x^{-1})^{-1}$ if $\A^{-}$ nontrivially intersects $\im(f)$. (These manifestly agree if $\im(f)$ intersects \textit{both} $\A^+$ and $\A^-$.) The map $(f: T \to \P^1_{\log}) \mapsto (u \in M_T^{\gp})$ defines a morphism 

\[
\P_1^{\log} \to \G_m^{\log}.
\]

This map is in fact a log \'etale cover; though it is \textit{not} Kummer; rather, it is a so-called ``log modification" (in other words a log blowup -- see \cite{MolchoWise2018TheLP}, Lemma 2.2.7.3.). Log blowups are in fact acyclic (\cite{IllusieLogEt}, Theorem 6.2); $\G_m^{\log}$ has the same homotopy type as $\P^1_{\log}$, which, over an algebraically closed field of characteristic 0, has $\pi_1^{\log}$ given by $\widehat{\Z}$; for algebraically closed fields of characteristic $p$, the $\pi_1^{\log}$ is given by $\prod_{\ell \ne p} \Z_\ell$. In short, the $\pi_1^{\log}$ of $\G_m^{\log}$ looks very much like the $\pi_1^{et}$ of $\G_m$ (and so, in turn, like the homotopy type of the circle $S^1$). 

Indeed, it is readily seen that the log toric variety $\P_1^{\log}$ has a Kato-Nakayama space \cite{katonakayama1999log} given by a cylinder $[0,1] \times S^1$; the corresponding projection map $\P^1_{\log} \to \G_m^{\log}$ projects onto $S^1$ (thereby remembering only the logarithmic, or boundary data of $\P^1_{\log}$). In other words, we may visualize $\G_m^{\log}$ as an $S^1$, which has the same homotopy type as $\G_m$ (at least in characteristic 0). 

We close by mentioning a pleasant fact which will become useful later. Let $\F_q^{\log}$ be the standard log point over $\F_q$; that is, the fs log scheme associated to the constant monoid $\N$ on $\Spec(\F_q)$, with $\alpha 1 \mapsto 0$. We obviously have the inclusion of the origin: $\F_q^{\log} \to \P_1^{\log}$ (an inclusion which does \textit{not} factor through an $\F_q$ point).

\begin{prop}\label{tamegroupisos}
    The maps $\F_q^{\log} \to \P^1_{\F_q, \log} \to \G_{m, \F_q}^{\log}$ induce canonical isomorphisms of $\pi_1^{\log}$. The resulting fundamental group is isomorphic to $\widehat{\Z(1)}' \rtimes_q \widehat{\Z}$, where $\widehat{\Z(1)}' = \underset{(n, q)=1}{\underset{\longleftarrow}{\lim}} \mu_n$, and $\rtimes_q$ means that if $F$ is the topological generator of $\widehat{\Z}$ and $x \in \widehat{\Z(1)}'$, then $FxF^{-1} = x^q$. In particular, $\pi_1^{\ab} \simeq \mu_{q-1} \times \widehat{\Z}\simeq \F_q^{\times} \times \widehat{\Z}$.
\end{prop}

\section{Compactified Vector Spaces}\label{compvecsp}

\subsection{Compactified Vector Spaces over \texorpdfstring{$k$}{k}} The centerpiece of Deligne's argument for geometric class field theory is the simply connectedness of $\P^N$. 
For us, the role of $\P^N$ will be played by logarithmically-compactified vector spaces; that is, $\P^N$ with log structure at a hyperplane divisor. We shall presently discuss the basic properties of such log schemes.

If $V \subset P(\widetilde{V})$ is an affine open embedding of a vector space into a projective space, with $P(\widetilde{V})\setminus V = H$, we say that the fs log scheme $P(\widetilde{V})$ with divisorial log structure at $H$ is a logarithmic compactification of $V$, and we will call such a log scheme a \textit{compactified vector space}.

One functorial formulation of a log-compactified vector space, which is helpful for computations, and which appears (at least for the case of line bundles) in \cite{FukKatShar} is:

\begin{prop}
    The functor of points of a logarithmically-compactified vector space $\overline{V}^{\log}$ is given by

\[
V \cup_{V^{\times}} \left(M^{-1} \times^{\mathcal{O}^{\times}} V^{\times}\right); 
\]

\noindent that is,

\[
(\Spec \,R, R, M, \alpha) \mapsto V(R) \cup_{V^{\times}(R)} (M^{-1} \times^{R^{\times}} V^{\times}(R)),
\]

\noindent where: 

i) $V(R) =  \{(r_1, \ldots r_n) \in R^n\}$; $V^{\times}(R) =  \{(r_1, \ldots r_n) \in R^n \,: \, r_1, \ldots, r_n$ generate the unit ideal in R$\}$;

ii) $M^{-1} = \{m^{-1}\in M^{gp} \,:\, m \in M\}$;

iii) $(M^{-1} \times^{R^{\times}} V^{\times}(R)) = \left(M^{-1} \times V^{\times}(R) \right)/ \sim$, where

iv)  $\sim$ is the equivalence relation $(tu, v) \sim (t, uv)$, for $t \in M^{-1}$, $u \in R^{\times}$, $v \in V^{\times}(R)$; and 

v) where $\cup_{V^{\times}}$ means the union, identifying $V^{\times}(R) \subset V(R)$ with $V^{\times}(R) \subset M^{-1} \times^{R^{\times}} V^{\times}(R)$.
\end{prop}

Notice that the set-theoretic output of a log-compactified vector space is not readily identifiable as a group (or a monoid, or any other familiar structure). Rather it is a strange amalgam of additive and multiplicative structures; we can say that its ``body" $V$ is additive, while its ``head" $\left(M^{-1} \times^{\mathcal{O}^{\times}} V^{\times}\right)$ is multiplicative. Professor Kato, for this reason, likens a log-compactified vector space to the centaur of Greek mythology, which has the head of a man and the body of a horse \cite{FukKatShar}. 

Another formulation, which will naturally occur in our study of higher-degree analogues of the Abel-Jacobi map, is as follows.
 
\begin{prop}\label{CompactVectPt}
    Let $V$ be a vector space, and $\varphi: V \to k$ a nonzero linear functional. Consider the presheaf $F$ such that, for a log scheme $T$ over a local ring, the set $F(T)$ is given by: 
\[
F(T):=\{[V\otimes \mathcal{O}_T]^\times \times_{\mathcal{O}_T} \mathcal{M}_T\}/ \mathcal{O}_T^{\times},
\]

\noindent defined in detail as follows. The set $[V\otimes \mathcal{O}_T]^\times$ means elements of $V \otimes\mathcal{O}_T$ that are nonzero under specializing to any point $t \in T$ (equivalently, if we write $e_i$ for a $k$-basis of $V$, then $\sum_i e_i \otimes f_i \in V \otimes \mathcal{O}_T$ lies in $V \otimes \mathcal{O}_T^{\times}$ if and only if the $f_i$ generate the unit ideal in $\Gamma(T, \mathcal{O}_T)$). The $\times_{\mathcal{O}_T}$ means fiber-product, taken with respect to the maps $\varphi \otimes 1$, and $\alpha$. Finally, the quotient is taken with respect to the diagonal action of $\mathcal{O}_T^{\times}$. 

Then the sheafification of the functor $F$ represents the compactified vector space $\P(V)$ with log structure at $\P(\ker(\varphi))$.
\end{prop}

\begin{proof}
    A global section of $[(a, m)] \in \{[V\otimes \mathcal{O}_T]^\times \times_{\mathcal{O}_T} \mathcal{M}_T\}/ \mathcal{O}_T^{\times}$ (for $a \in \Gamma ([V\otimes \mathcal{O}_T]^\times)$ and $m \in \Gamma(\mathcal{M}_T)$), specifies the data of a map $a: T \to V\setminus\{0\}$, which is only defined up to scaling. This gives us a map $T \to \P(V)$. The $m$ gives a section of $\mathcal{M}_T$, well-defined over each standard affine-open, that is trivial (i.e., in $\mathcal{O}_T^{\times}$) over the locus $a^{-1}\left(\P(V) \setminus\P(\ker \varphi)\right)$. We note that $\left(V\setminus\{0\}\right)_{\log} \to \P(V)(\log \P(\ker(\varphi)))$ (the former with log structure strictly pulled back from that of $\P(V)(\log \P(\ker\varphi))$) is a $\G_m$ torsor, and $\G_m$ torsors are trivial over local rings. Thus, for local $T$ we may lift $T \to \P(V)(\log \P(\ker\varphi))$ to $T \to (V\setminus\{0\})(\log(\ker(\varphi)\setminus\{0\}))$, giving us an element in $\{[V\otimes \mathcal{O}_T]^\times \times_{\mathcal{O}_T} \mathcal{M}_T\}$, which descends to a well-defined element of $F(T)$.
\end{proof}

We shall see that this is the particular functorial formulation of compactified vector spaces that naturally occur in our story.

\subsection{Families of Compactified Vector Spaces} It will be important for us to adopt Grothendieck's relative points of view; that is to say, we ought to upgrade the notion of a log-compactified vector space from a  property of an individual log scheme to a property of a \textit{morphism} between two log schemes -- or even to a morphism between two functors $\textbf{Log Scheme}^{\op} \to \textbf{Sets}$. In other words, we wish to define the notion of a ``relative"  compactified log vector-bundle. 

Let $T$ be an fs log scheme. Then a compactified vector bundle over $T$ (or a $T$-family of log-compactified vector spaces) is given by a morphism between fs log schemes

\[
\pi: \mathcal{E} \to S
\]

\noindent which has the following properties: 

i) The underlying scheme $\mathcal{E}$ is a $\P^N$-bundle over $S$ -- that is to say: the map $\mathcal{E} \to S$ is flat with all geometric fibers isomorphic to $\P^N$;

ii) there is a horizontal divisor $H \subset \mathcal{E}$ such that, for all geometric points $s \in S$, the fiber $H|_{s}$ is given by a (linear) hyperplane $\P^{N-1}|_s \subset \P^N|_s$;

iii) the log structure of $\mathcal{E}$ is isomorphic to the pushout 

\[\pi^*(\mathcal{M}_S) \times^{\mathcal{O}_\mathcal{E}^{\times}} \mathcal{M}_H,
\] 

\noindent where $\pi^*(\mathcal{M}_S)$ is the strict pullback of the log structure from $S$ to $\mathcal{E}$ and $\mathcal{M}_H$ is the sheaf of monoids associated to the divisorial log structure of $H$.

Moreover, if we let $\phi: F\to G$ be a natural transformation between two functors  

\[
F,G: \textbf{Log Schemes}^{\op} \to \textbf{Sets},
\]

\noindent then we say that $\phi$ is a \textit{relative compactified vector space} or once more a \textit{compactified vector bundle} if, for all affine fs log schemes $T$, and morphisms $t : T \to G$ (viewing $T$ as a functor $\textbf{Log Schemes}^{\op} \to \textbf{Sets}$ and $t$ as a natural transformation between $\phi$ and $G$), the fiber product functor:

\[
\{t\} \times_G {\phi}: \textbf{Log Schemes}^{\op} \to \textbf{Sets}
\]

\noindent is representable by a compactified vector bundle over $T$.

\subsection{The Simply Connectedness of Log-Compactified Vector Spaces}  In this section we will prove:

\begin{prop}\label{simpleconnfiber}
    Let $S$ be a logarithmic stack and $p: \mathcal{E} \to S$ be a compactified vector bundle. Then $\pi_1^{\log}(\mathcal{E}) \to \pi_1^{\log}(S)$ is an isomorphism. 
\end{prop}

\begin{proof}

 We observe that $p: \mathcal{E} \to S$ is a log-proper, log-regular fibration. By the homotopy exact sequence of $\pi_1^{\log}$'s for a fibration (see Theorem \ref{loghtpyseq}, Section \ref{htpylogstack}, and \cite{Hoshi2009LogHomotopy}) the theorem boils down to the claim that $\P^N$ with divisorial log structure on a hyperplane $\P^{n-1} \subset \P^N$ is (log)-simply connected. By \ref{FundIso}, this is equivalent to showing that there are no tamely ramified covers of $\P^N$ with branching locus lying over a hyperplane, which is a classical result (see, e.g., \cite{toth2011geometric}, example 3.3.8). 
\end{proof}

\begin{remark}\label{contractibility}
    It is interesting to note that this even holds in characteristic $p$; where the non-compactified $\A^1_p$ is famously \textit{not} simply connected. In fact, the map $\mathcal{E} \to S$ should be a log-homotopy equivalence; that is to say, $\P^N$ with log structure on a hyperplane is \textit{contractible}. Thus the natural map $\mathcal{F} \to Rp_*p^*(\mathcal{F})$ is an equivalence for any $\mathcal{F}\in D^+(S, \Lambda)$, with $D^+$ the bounded-below derived category of constructible sheaves with respect to the Kummer \'etale  topology and $\Lambda = \Z/d\Z$ (including $d$ not invertible on $S$). This is interesting geometrically, but will not be required in what follows. 
\end{remark}

\section{The (Rigidified) Logarithmic Picard and Jacobian}\label{LogPicRig}

\subsection{The Classical Theory}\label{ClassicalRigPic} (cf. \cite{toth2011geometric}). Let $X$ denote a proper, smooth, geometrically connected curve over a field $k$. Let $S$ denote a finite closed set of points in $X$. For simplicity, let us assume that all are $k$-rational.

The \textit{classical tamely ramified Picard scheme}, denoted $\Pic_{X,S}$, represents the functor 

\[
T \mapsto \{(\mathcal{L}, \theta)\}/\cong,
\]

\noindent where $\mathcal{L}$ is a line bundle on $X \times T$ and $\theta = (\theta_s)_{s \in S}$ is a collection of isomorphisms

\[
\theta_s: \mathcal{O}_T \xrightarrow{\sim} i_s^*\mathcal{L},
\]

\noindent with $i_s :  T \to X \times T$ the map $(s, \Id)$, for $s$ the constant map $T \to \{s\} \in X$. Here $\cong$ means isomorphisms $\mathcal{L} \xrightarrow{\sim} \mathcal{L}'$ intertwining the corresponding isomorphisms $\theta_s$. We may also think of this as the data of a global section of $i_s^*\mathcal{L}$ over $T$. 

Another construction for $\Pic_{X,S}$ is given by the functor:

\begin{equation}\label{H1Def}
 T \mapsto H^1(X \times T, \,\Ker (\mathcal{O}_{X \times T}^{\times} \to \oplus_s i_{s_*} \mathcal{O}_T^{\times})).
\end{equation}

There is an obvious forgetful map $\Pic_{X,S} \to \Pic_X$. The fibers are given by the collection of isomorphisms $\theta_s$. These isomorphisms are acted on by $\mathbb{G}_m^{|S|}$ -- though this action is not free; indeed, note that if $(\mathcal{L}, \mathcal{\theta})$ is a $T$-point of $\Pic_{X,S}$, then for all $u \in \mathbb{G}_m(T) = \mathcal{O}_T^{\times}$, we may define $u\theta :=(u\theta_s)_{s \in S}$,

\[
u\theta_s': \mathcal{O}_T \xrightarrow{\mu_{u^{-1}}} \mathcal{O}_T \overset{{\theta_s} }{\longrightarrow} i_s^*\mathcal{L}
\]

\noindent where $\mu_{u^{-1}}$ denotes multiplication by $u^{-1} \in \mathcal{O}_T^{\times}$. But clearly $(\mathcal{L}, \theta) \cong (\mathcal{L}, u\theta)$. Thus the ``diagonal" action of $\mathbb{G}_m$ on the $\theta_s$ is trivial, and we see that the map $\Pic_{X,S} \to \Pic_X$ is a $\mathbb{G}_m^{|S|-1}$ torsor.

In other words, once we pick an identity $\theta_{0} = (\theta_0)_S$, we may identify the fibers of  $\Pic_{X,S} \to \Pic_X$ with a torus  $\mathcal{T} = \left(\oplus_{s \in S} \mathbb{G}_m\right)/\mathbb{G}_m$, where $\mathbb{G}_m$ is embedded diagonally. (If we had not assumed each $s \in S$ to be rational, then we would use $\left(\Res_{S/k}(\mathbb{G}_m)\right)/\mathbb{G}_m$ instead.) If $\mathcal{L}_0 \in \Pic_{X,S}(k)$ is the trivial line bundle (i.e., $\mathcal{O}_X$), then $\Pic_{X,S}$ is canonically identified with $\mathcal{T}$ and we have an exact sequence:

\begin{equation}\label{TorSeq}
\begin{tikzcd}
1 \arrow[r] &\mathcal{T} \arrow[r] & {\textrm{Pic}_{X,S}} \arrow[r] & \textrm{Pic}_X \arrow[r] & 1.
\end{tikzcd}
\end{equation}

\noindent Notice that the classical $\Pic_{X,S}$ is the extension of $\Z$ by a semi-abelian variety; $ \Pic_{X,S}^0$ is the extension of the abelian variety $J_{X}$ by $\G_m^{|S|-1}$. 

Recall that a semi-abelian variety $G$ (living as an extension $1 \to T \to G \to A \to 1$) is automatically a commutative group scheme: the conjugation action $\textrm{Inn}: G \to \Aut(T) = \GL_{|S|-1}(\Z)$, has constant image since $G$ is connected, whence $G$ acts trivially by conjugation on $T$. Hence $T$ is central in $G$. On the other hand, the commutator map $c: G \times G \to G; (g, h) \mapsto ghg^{-1}h^{-1}$, maps to the identity in $A$ under projection since $A$ is commutative. Thus the image of the commutator map must lie in $T = \ker(G \to A)$. Then $c$ factors through $A$, since $T$ is central, giving a map $A \times A \to T$. But $A\times A$ is proper while $T$ is affine, so the image is constant: it must be the identity. 

From this we conclude that $\Pic_{X,S}$ is a commutative group scheme.

\subsection{The (Framed) Logarithmic Jacobian and Picard} $\Pic_{X,S}$ is not proper -- indeed, the torus $\mathcal{T}$ in (\ref{TorSeq}) 
is not proper. This is highly undesirable for a moduli space. Luckily, log geometry comes to our rescue: we may compactify $\mathcal{T}$ using $\mathbb{G}_{m }^{\log}$. Observe that, as a functor of points:

\[
\mathcal{T} = \mathcal{H}om(\mathcal{H}om(\mathcal{T}, \G_m), \G_m).
\]

\noindent So we define:

\begin{equation}
\overline{\mathcal{T}} = \mathcal{H}om(\mathcal{H}om(\mathcal{T}, \G_m), \G_m^{\log}),
\end{equation}

\noindent and define $\overline{\Pic}_{X,S}$ via the pushout diagram:

\begin{equation}\label{logpushout}
\begin{tikzcd}
1 \arrow[r] & \mathcal{T} \arrow[d, hook] \arrow[r] & {\textrm{Pic}_{X,S}} \arrow[d, hook] \arrow[r] & \textrm{Pic}_{X} \arrow[d] \arrow[r] & 1 \\
1 \arrow[r] & \overline{\mathcal{T}} \arrow[r]      & {\overline{\textrm{Pic}}_{X,S}} \arrow[r]      & \textrm{Pic}_{X} \arrow[r]           & 1
\end{tikzcd}
\end{equation}

\noindent We call this the ``framed" or ``rigidified" logarithmic Picard; sometimes we will simply refer to it as the logarithmic Picard. We can also provide a description of the functor of points of $\overline{\Pic}_{X,S}$ akin to (\ref{H1Def}): 

\begin{equation}
(T, \mathcal{O}_T, \mathcal{M}_T) \mapsto H^1( X \times T, \,\Ker[\mathcal{O}_{X\times T}^{\times}M_T^{\gp} \to \oplus_{s \in S} i_{s*}(M_T^{\gp})]).
\end{equation}

\noindent Here $\mathcal{O}_{X,T}^{\times}M_T^{\gp} = \mathcal{O}_{X\times T}^{\times} \times^{\mathcal{O}_T^{\times}}M_T^{\gp}$; equivalently, $\mathcal{O}_{X\times T}^{\times}M_T^{\gp}  = (p_1^*(M_T))^{\gp}$, where $p_1: T  \times X \to X$ is the projection onto the first factor. Both of these descriptions of $\overline{\Pic}_{X,S}$ demonstrate that it is naturally a commutative group-valued functor on the category of log schemes.

We may also give a more concrete description of $\overline{\Pic}_{X,S}$ as a functor of points:

\begin{equation}
(T, \mathcal{O}_T, \mathcal{M}_T)  \mapsto \{(\mathcal{L}, \theta)\}/\cong
\end{equation}

\noindent where $\mathcal{L}$ denotes an isomorphism class of line bundles on $X \times T$, $\theta = (\theta_s)_{s\in S}$ is a family of sections $\theta_s$ of 

\[
i_s^*(\mathcal{L}^{\times})\times^{\mathcal{O}_T^{\times}} M_T^{\gp},
\]

\noindent and $(\L, \theta) \cong (\L',\theta')$ if there exists a section of $\mathcal{I}som(\L, \L') \times^{\mathcal{O}_T^{\times}}M_T^{\gp}$ which sends $\theta$ to $\theta'$. (Of course, it would be more functorially parsimonious to include the data of such an isomorphism as part of the structure of $\overline{\Pic}_{X,S}$; following this idea leads us to the notion of the logarithmic Picard \textit{stack}, a lax functor $\textbf{Log Schemes}^{\op} \to \textbf{Gpd}$.)

Observe that we still have a natural forgetful map $\overline{\Pic}_{X,S} \to \Pic_X$ (indeed, we have already seen it in the bottom line of (\ref{logpushout})). We write $\overline{\Pic}^n_{X,S}$ or more simply $\overline{J}^{n}_{X,S}$ for the inverse image of $n \in \Z$ under the composite

\[
\begin{tikzcd}
{\overline{\textrm{Pic}}_{X,S}} \arrow[r] & \textrm{Pic}_X \arrow[r, "\textrm{deg}"] & \mathbb{Z}.
\end{tikzcd}
\]

\noindent And we write:

\[
\overline{J}_{X,S}:= \overline{J}_{X,S}^{0},
\]

\noindent which is the logarithmically compactified rigidified Jacobian with respect to $S$.

\begin{remark}
    We note the relationship between the logarithmic Picard scheme ``$\log\Pic$" of Molcho and Wise \cite{MolchoWise2018TheLP} and our $\overline{\Pic}_{X,S}$. Molcho and Wise consider the moduli space of $\G_m^{\log}$-torsors (with bounded monodromy) over an arbitrary log scheme. Our $\overline{\Pic}_{X,S}$ is akin to the Molcho-Wise $\log\Pic$ for the specific case of a smooth curve with divisorial log structure at $S$, with additional rigidification data.
\end{remark}

\subsection{Properness of the Logarithmic Jacobian} We observe here that 

\[
\overline{\Pic}_{X,S} \to \Pic_{X} 
\]

\noindent exhibits $\overline{\Pic}_{X,S}$ as a $(\G_m^{\log})^{|S|-1}$-torsor over $\Pic_{X}$.
Observe that log properness, defined by lifting properties of maps (from log annuli) \textit{into} a log-space, is preserved under fiber products. This means that, given a fibration with both fibers and base log-proper, the total space is also log proper. Applying Propositions \ref{logproperscheme} and \ref{logproperGm}, we obtain:

\begin{prop}\label{propernessJac}
    For each $n$, $\overline{J}_{X,S}^n$ is log proper. In particular $\overline{J}_{X,S} =\overline{J}_{X,S}^0$ is a log-proper abelian group-valued functor on the category of fs log schemes.
\end{prop}

\begin{remark} 
Note that this is a manifestation of the ``magic" of log geometry: $\overline{\Pic}_{X,S}$ simultaneously compactifies the scheme $\Pic_{X,S}$ \textit{and} maintains its group structure, something impossible in the world of schemes. 
\end{remark}

\begin{remark}
    By construction, the log-Jacobian $\overline{J}_{X,S}$ is a $\left(\G_m^{\log}\right)^{|S|-1}$-torsor over the classical Jacobian $J$ of the curve. Applying the homotopy exact sequence and our discussion of the homotopy type of $\G_m^{\log}$ (see Section \ref{htpylogstack}), we can see the similarities between the $\pi_1$ of $\overline{J}_{X,S}$ (resp $J_{X,S}^n$) and the classical tamely ramified $\Pic_{X,S}$ (resp $\Pic_{X,S}^n$). Thus our logarithmic upgrading of $\Pic_{X,S}$ does not radically alter homotopy type; this will be an important sanity-check in comparing the logarithmic version of geometric class field of theory with its classical tamely-ramiified counterpart.
\end{remark}

\subsection{Character Sheaves on \texorpdfstring{$\overline{\Pic}_{X,S}$}{PicXS}} 

In this subsection, we shall work over the base field $k = \F_q$. We have seen that $\overline{\Pic}_{X,S}$ is a proper commutative group object in the category of log spaces. Recall that in the world of schemes, there is a canonical correspondence between 1-dimensional multiplicative $\overline{\Q}_{\ell}$-adic local systems over a commutative $\F_q$-group-scheme $G$ (i.e., character sheaves on $G$) and characters $G(\F_q) \to \overline{\Q_{\ell}}^{\times}$. Here we will derive an analogous correspondence for the log space $\overline{\Pic}_{X,S}$. 

Firstly, we review terminology. Let $G$ be a commutative group object in the category of logarithmic spaces, with $m : G \times G \to G$ the multiplication map. Let $0 \in G$ be the identity (that is, the subfunctor of $G: \textbf{Log Sch}^{\op} \to \textbf{Ab}$ such that $0(T)$ is the singleton set consisting of the identity in $G(T)$ for all $T$). We pick a log geometric point $\overline{0}$ above $0$. A 1-dimensional multiplicative local system (or a ``character sheaf") $\mathfrak{L}$ on $G$ is a local system $\mathfrak{L}$ on $G$, with a rigidification $\mathfrak{L}_{\overline{0}} \simeq \overline{\Q_{\ell}}$, and an isomorphism $\mu: m^{-1}(\mathfrak{L}) \to \mathfrak{L} \boxtimes \mathfrak{L}$.

For a log scheme $X$ defined over $\F_q$, we recall that the absolute Frobenius $F: X \to X$ is given by the identity on topological spaces, the map $f \mapsto f^q$ on functions, and $m \mapsto m^q$ (or $qm$ written additively) on monoids. Absolute Frobenius commutes with arbitrary morphisms, so for a logarithmic space $\mathfrak{X}$ over $\F_q$, we may define Frobenius as the natural transformation $\mathfrak{X}(T) \to \mathfrak{X}(T)$ given by $\phi \mapsto F_T^*(\phi)$. For instance, $F: \G_m^{\log}\to\G_m^{\log}$ is the $q$th power map $\mathcal{M}_T^{\gp} \to \mathcal{M}_T^{\gp}$; $x \mapsto x^q$ on $T$-points. 

Moreover, we have a natural faisceaux-fonctions correspondence: given a local system $\mathfrak{L}$ on $\mathfrak{X}$ and an $\F_q$-point $x \in \mathfrak{X}(\F_q)$ with $\Spec(\F_q)$ given trivial log structure, we have a well-defined map $\pi_1^{et}(\Spec(\F_q), \overline{x}) \to \pi_1^{\log}(\mathfrak{X}, \overline{x})$ (for $\overline{x}$ any geometric point of $\Spec(\F_q)$). The image of Frobenius gives us an element of $\pi_1(\mathfrak{X}, \overline{x})$; a local system is equivalent to a representation $\pi_1(\mathfrak{X}, \overline{x}) \to \End_{\overline{\Q}_{\ell}}(V)$ for $V$ a $\overline{\Q}_{\ell}$-space; this is well-defined up to conjugacy. Thus we may take the trace, giving a well-defined function $\mathfrak{X}(\F_q) \to \overline{\Q_{\ell}}$.

As in the category of commutative (non-log) group schemes over $\F_q$, we have the Lang isogeny $L: G \to G$, $g \mapsto F(g)g^{-1}$ which is well-defined for log spaces. For smooth $G$ this map is \'etale, and exhibits $G$ as a covering space of itself.

\begin{prop}\label{MultSysGmlog}
   Multiplicative local systems on $\G_m^{\log}$ are in bijection with characters of  $\F_q^{\times} = \G_m^{\log}(\F_q) = \ker\left(L_{\G_m^{\log}}\right)$. 
\end{prop}

\begin{proof} We know from Proposition \ref{tamegroupisos} that $\pi_1^{\log}(\G_m^{\log}, \overline{x}) = \widehat{\Z(1)}'\rtimes_q \widehat{\Z}$ and $\pi_1^{\log, \ab} (\G_m^{\log})$ is $\F_q^{\times} \times \widehat{\Z}$. Similarly we compute $\pi_1^{\log} (\G_m^{\log} \times_{\F_q} \G_m^{\log}) = \left( \widehat{\Z(1)}'\times \widehat{\Z(1)}'\right)\rtimes_q \widehat{\Z}$ and $\pi_1^{\log, \ab} (\G_m^{\log} \times_{\F_q} \G_m^{\log}) = \F_q^{\times} \times \F_q^{\times} \times \widehat{\Z}$. Multiplication $m: \G_m^{\log} \times \G_m^{\log} \to \G_m^{\log}$ induces the map $\F_q^{\times} \times \F_q^{\times}\times \widehat{\Z} \to \F_q^{\times} \times \widehat{\Z}$ given by $(x,y; z)\mapsto (xy; z)$, while the two projections correspond to $(x, y; z) \mapsto (x; z)$ and $(x,y;z)\mapsto (y;z)$.
Given a character $\chi: \pi_1^{\log, \ab} (\G_m^{\log}) \to \overline{\Q_{\ell}}^{\times}$. We see that $m^*(\chi) = \pi_1^*(\chi) \cdot \pi_2^*(\chi)$  iff $\chi$ is trivial on $\widehat{\Z}$; that is, a character on $\F_q^{\times}$. 
\end{proof}

From this we deduce:

\begin{prop}\label{CharsheavesGmlog}
   Character sheaves on $\overline{\Pic}_{X,S}$ correspond to characters of $\overline{\Pic}_{X,S}(\F_q)$.
\end{prop}

\begin{proof}
    We perform a series of reductions. First, let us base change to $\overline{\F_q}$. Then we have the exact sequence 

    \[
    1 \to \overline{T} \to \overline{J}_{X,S} \to J_X \to 1.
    \]

    \noindent where $\overline{T} = \left(\G_m^{\log}\right)^{|S|-1}$. Let $N$ be relatively prime to $q$; the multiplication-by-$N$ map is finite \'etale and surjective on $\overline{T}$. Then by the snake lemma applied to the vertical kernels and cokernels of the diagram

    \[
\begin{tikzcd}
1 \arrow[r] & \overline{T} \arrow[r] \arrow[d, "N"] & {\overline{J}_{X,S}} \arrow[r] \arrow[d, "N"] & J_{X} \arrow[r] \arrow[d, "N"] & 1 \\
1 \arrow[r] & \overline{T} \arrow[r]                & {\overline{J}_{X,S}} \arrow[r]                & J_{X} \arrow[r]                & 1,
\end{tikzcd}
    \]

    \noindent we find that

    \[
    1 \to \overline{T}[N] \to \overline{J}_{X,S}[N] \to J_X[N] \to 1
    \]

    \noindent is exact. Moreover, this is an exact sequence of finite group schemes so $\Ext^1(J_X[N], \overline{T}[N])$ is trivial; we conclude that this sequence splits for all $N$. Hence $\underset{\underset{N}{\longleftarrow}}{\lim}\overline{J}_{X,S}[N] = \left(\widehat{\Z(1)}'\right)^{|S|-1} \times \widehat{J_X(\overline{\F_q})_{\textrm{tors}}}$, and we recognize this as $\pi_1^{\log}\left({\overline{J}_{X,S, \overline{\F_q}}}, \overline{1}\right)$ for a log geometric point lying above the identity.

    By the exact sequence for $\pi_1$ under base change of a field, we see that 

    \[
    \pi_1\left({\overline{J}_{X,S}}\right) = \left( \left(\widehat{\Z(1)}'\right)^{|S|-1}  \times \widehat{J_X(\overline{\F_q})_{\textrm{tors}}}\right)\rtimes\widehat{\Z}
    \]

    \noindent where the topological generator of the last factor acts via Frobenius on the factors in parentheses. Thus 
    
    \begin{equation}\label{pi1compute}
    \pi_1^{\log,\ab}\left(\overline{J}_{X,S}\right) = \left(\mu_{q-1}^{|S|-1} \times \widehat{J_X(\overline{\F_q})_{\textrm{tors}}}\right)\times\widehat{\Z}.
    \end{equation}
    
    Repeating the argument above, we see that a 1-dimensional multiplicative local system on $\overline{J}_{X,S}$ must correspond to a character of \ref{pi1compute} that kills the $\widehat{\Z}$-factor; i.e., it is a character of $\overline{J}_{X,S}(\F_q)$.

    Finally, we note that $\Pic_{X,S}$ is a (split) extension of $\Z$ by $\overline{J}_{X,S}$; the lemma follows.

\end{proof}

\section{The Map \texorpdfstring{$\psi_n: X^{(n)}_{\log} \to \overline{J}_{X,S}^{n}$}{psin: X(n)log to JXSn}}\label{GeomSymPic}

Recall that Deligne's proof of geometric class field theory boils down to scrutinizing the fibers of the map $X^{(n)} \to \Pic^n_X$ (which are $\P^N$'s). In the classical tamely ramified story (\cite{toth2011geometric}), the geometry of the map $U^{(n)} \to \Pic_{X,S}$ plays the central role (with $U = X \setminus S$). In particular, one must construct a relative compactification of $U^{(n)} \to \Pic_{X,S}$ (loc. cit.). We have already upgraded the target of this map to a logarithmic compactification of the classical rigidified Jacobian $\Pic_{X,S}$. We would now like to upgrade the source, too, so that the map is proper and so that, for sufficiently large $n$, $X^{(n)}_{\log} \to \overline{J}^{n}_{X,S}$ is surjective. In other words, we would like to compactify $U^{(n)}$. 

To leverage the connection between \'etale fundamental groups of log schemes and tamely ramified \'etale fundamental groups (see Corollary \ref{FundIso}), we must have $X^{(1)}_{\log(S)} = X(\log S)$. (Indeed, by Corollary \ref{FundIso}, we see that local systems on $X(\log S)$ correspond to representations of the tamely ramified fundamental group. Thus local systems on $X(\log S)$ are what will provide our connection with the tamely ramified local systems on $X$.)

One approach would be to formulate a logarithmic analogue of $\Div^n(X (\log S))$, the degree $n$ divisors of $X(\log S)$; in other words to describe the Hilbert scheme of points of $X(\log S)$. Recently, the notion of the logarithmic Hilbert scheme (and, more generally, the logarithmic Quot scheme) of an fs log scheme was defined by Maulik-Ranganathan \cite{MaulikRanganathan_LogDT_2024} and (in full generality) by Kennedy-Hunt \cite{kennedyhunt}. The geometry of the map $\text{Hilb}^n(X(\log S)) \to \overline{\Pic}_{X,S}$ given by this formulation will be studied in a later paper.

For our purposes, the analogue of the symmetric powers of $X$ in Deligne's argument will be less sophisticated: we will use the scheme $X^{(n)}$ with a divisorial log structure induced by $S$. We now make this precise.

\subsection{The Definition of \texorpdfstring{$X^{(n)}_{\log}$}{X(n)logS}} Let $X^{(n)}$ denote the $n$th symmetric power of $X$, and let $S^{(n)}$ denote the complement of $U^{(n)}$ in $X^{(n)}$. Let $r = |S|$. We define the divisorial log scheme $X^{(n)}_{\log(S)}$ to be $X^{(n)}(\log S^{(n)})$. We will usually write this as $X^{(n)}_{\log}$ for simplicity. Implicit in this definition is the following:

\begin{prop}\label{normcrosslem} $S^{(n)}$ is a normal crossings divisor in $X^{(n)}$.
\end{prop}

\begin{proof}
    Let us first check that this holds for the case of $(\A^1)^{(n)}$. Observe that $(\A^1)^{(n)} \simeq \Spec\, k[x_1, \ldots, x_n]^{S_n} = \Spec \, k[ \sigma_1 , \ldots, \sigma_n] \simeq \A^n$, where $\sigma_i$ is the $i$th elementary symmetric polynomial in the $x_i$. Abusing notation slightly, let us say $s \in \A^1$ is a single closed point corresponding to the element $s \in k$. The divisor $(\{s\})^{(n)}$ in $(\A^1)^{(n)}$ is then cut out by the equation

    \[
    s^n + s^{n-1} \sigma_1 + \cdots +\sigma_n = 0,
    \]

    \noindent which is simply a linear affine hyperplane in $(\A^1)^{(n)} \simeq \A^n$. Thus $S^{(n)}$ is a union of distinct affine linear hyperplanes, which is manifestly a normal crossings divisor. 

    The general case may be reduced to this one, since (by Noether normalization) a curve is Zariski-locally \'etale over $\A^1$.
\end{proof}

\subsection{The Functor of Points of \texorpdfstring{$X^{(n)}_{\log}$}{X(n)log}}

We would like to describe the functor of points of $X^{(n)}_{\log}$; i.e., to describe the set of maps $T \to X^{(n)}_{\log}$ for a test log scheme $T$.

Firstly, observe that if we pick $\varpi_s$ to be a local uniformizer of $s$ in $X$ for each $s \in S \subset X$, then $\varpi_s^{(n)} := \varpi_s \otimes \cdots \otimes \varpi_s \in \mathcal{O}\left(X^n\right)^{S_n}$ (product taken $n$ times) is a uniformizer for the divisor $\{s\}^{(n)}$ in $X^{(n)}$. Let $V_s \subset X$ be the largest open set of $X$ such that $\varpi_s$ is regular and nonvanishing except at $s$; observe that $\varpi_s \in \mathcal{O}_X(V_s)$, and $\varpi_s^{(n)} \in \mathcal{O}_{X^{(n)}}(V_s^{(n)})$ (that is to say, $\varpi_s$ is regular on $V_s$ and $\varpi_s^{(n)}$ is regular on $V_s^{(n)}$). Let us write $\pi_s^{(n)} \in \mathcal{M}_{X^{(n)}}\left(V_s^{(n)}\right)$ for the corresponding generator of the log structure $\mathcal{M}_{X^{(n)}}$ (that is, the local section of the monoid sheaf such that $\alpha\left(\pi_s^{(n)}\right) = \varpi_s^{(n)}$).

We observe that a map $T \to X^{(n)}_{\log}$ is given by the following data:

1) a map of schemes $\psi: T \to X^{(n)}$; equivalently, an unordered set of $n$ maps $\psi_i: T \to X$, for $i = 1, \ldots, n$.

2) Sections $\mu_s \in \mathcal{M}_T\left(\psi^{-1}\left(V_s^{(n)}\right)\right)$, for each $s \in S$, such that 

\[
\alpha(\mu_s) = \psi^*\left(\varpi_s^{(n)}\right) \in \mathcal{O}_T\left(\psi^{-1}\left(V_s^{(n)}\right)\right).
\]

\noindent Indeed: each $\pi_s \in \mathcal{M}_{X^{(n)}}\left(V_s^{(n)}\right)$ must be sent to an appropriate section of $\mathcal{M}_T$, and specifying these sections uniquely determines the map of log schemes $T \to X^{(n)}$. In other words, the functor $T \mapsto \{\psi; (\mu_s)_{s \in S}\}$ describes the functor of points represented by $X^{(n)}_{\log}$.

We would like to rephrase this in terms of line bundles and sections. Firstly, recall that the symmetric power $X^{(n)}$ has a functor of points described as follows:

\[
T \mapsto \{(\mathcal{L}, \sigma)\}/\sim,
\]

\noindent where 

i)  $\mathcal{L}$ is a line bundle on $T \times X$;

ii) $\sigma$ is a nonzero global section, such that 

iii) the vanishing locus of $\sigma$, which we shall call $Z_\sigma \subset T \times X$, is flat over $T$; 

iv) $\sim$ is the natural equivalence relation $(\mathcal{L}, \sigma)\sim(\mathcal{L}', \sigma') $ induced by isomorphisms $\mathcal{L} \xrightarrow{\sim} \mathcal{L}'$ carrying $\sigma$ to $\sigma'$. Equivalently, we identify $\mathcal{L}$ up to isomorphism and $\sigma$ up to $\mathcal{O}_T^{\times}$-scalar multiple.

Such a pair $(\mathcal{L}, \sigma)$ canonically corresponds to a map of schemes $\psi: T \to X^{(n)}$ (or an unordered set of $n$ maps $\{\psi_i: T \to X\}_{i=1, \ldots, n}$). As for the $\mu_s$, we claim that specifying them is equivalent to specifying sections

\[
\widetilde{\theta}_s \in \Gamma\left( \mathcal{M}_T \times^{\mathcal{O}_{T\times \{s\}}^{\times}} \left(\mathcal{L}^*|_{T \times\{s\}}\right)\right),
\]

\noindent for each $s \in S$, where $*$ means nonzero\footnote{Though observe that such a section $\sigma$ of $\mathcal{L}^*$ may actually restrict to 0 in $\mathcal{L}|_{T \times \{s\}}$. This occurs precisely when one of the maps $T \to X$ corresponding to $\sigma$ is the constant map into $s$.} sections of $\mathcal{L}$ (which we restrict to $T \times\{s\}$), such that the map 

\[
\mathcal{M}_T \times^{\mathcal{O}_T^{\times}} \left(\mathcal{L}^*|_{T \times\{s\}}\right) \to \mathcal{O}_T \otimes_{\mathcal{O}_T}\mathcal{L}|_{T \times\{s\}}
\]

\noindent induced by $(\alpha, 1)$, sends $\widetilde{\theta}_s \mapsto \sigma|_{T \times\, \{s\}}$.

\begin{prop}
$X^{(n)}_{\log}$ is represented by the functor 

\[
T \mapsto \{(\mathcal{L}, \, \sigma, \,(\widetilde{\theta}_s)_{s \in S})\}/\sim,
\]

\noindent where $\mathcal{L}$, $\sigma$, and $(\widetilde{\theta}_s)_{s \in S}$ are as above, and the equivalence relation $\sim$ is given as follows: for a fixed isomorphism class $\mathcal{L} \in \Pic_X(T)$, we say $(\mathcal{L}, \, \sigma, \, \widetilde{\theta}_S) \sim (\mathcal{L}, \, \lambda \sigma, \, \lambda\widetilde{\theta}_S)$ for all $\lambda \in \mathcal{O}_T^{\times}$.

\end{prop}

\begin{proof} To see this, we shall construct the maps in both directions. Let us begin with an unordered set of maps $\{\psi_i\}: T \to X$ and sections $\mu_s \in \mathcal{M}_T\left(\psi^{-1}\left(V_s^{(n)}\right)\right)$ such that $\alpha (\mu_s) = \psi^*\left(\varpi_s^{(n)}\right)$. From this we shall construct $(\mathcal{L}, \sigma, \widetilde{\theta}_S)$. The $n$ maps $\psi_i: T \to X$ correspond to the Cartier divisor on $T \times X$ given by the $\mathcal{O}_{T \times X}$-submodule of the total quotient $K(T \times X)$ generated by the collection:

\[
\prod_{i =1}^n \frac{1}{\psi_i^*(f) \otimes 1 - 1 \otimes f},
\]

\noindent for $f \in \mathcal{O}_X$. (Note that this is indeed a Cartier divisor; locally a single such $f$ generates this submodule since $X$ is a curve.) The section $\sigma$ corresponds to the rational function $1 \in K(T \times X$). This gives us $\mathcal{L}$ and $\sigma$, as in the classical story.

On the other hand, given $\mu_{s}$, we define $\widetilde{\theta}_{s}$ as

\begin{align*}
\widetilde{\theta}_s &= \mu_s \times^{\mathcal{O}_T^{\times}}\prod_{i =1}^n\frac{1}{\psi_i^*(\varpi_s)\otimes 1 -1 \otimes \varpi_s}\bigg\vert_{T \times \{s\}}\\
&= \mu_s \times^{\mathcal{O}_T^{\times}}\prod_{i =1}^n\frac{1}{\psi_i^*(\varpi_s)\otimes 1}\\
&= \mu_s \times^{\mathcal{O}_T^{\times}}\frac{1}{\psi^*\left(\varpi_s^{(n)}\right)\otimes 1}
\end{align*}

\noindent over the open set $\psi^{-1}\left(V_s^{(n)}\right)$ in $T$, while we define it as $1 \times^{\mathcal{O}_T^{\times}} 1$ on $\psi^{-1}(X^{(n)} - D_s)$. We see that these glue on the open set for which $\mu_s \in \mathcal{O}_T^{\times} \subset \mathcal{M}_T$. Moreover: 

\[
(\alpha,1)(\theta_s)= \psi^*(\varpi_s^{(n)}) \otimes_{\mathcal{O}_T} \frac{1}{\psi^*(\varpi_s^{(n)}) } = 1 = \sigma|_{T \times \{s\}},
\]

\noindent where we view $\sigma \in K(T \times X)$. This verifies the necessary constraint on $\widetilde{\theta}_s$. 

Conversely, given $(\mathcal{L}, \sigma, \{\widetilde{\theta}_s\}_{s \in S})$ we construct $\psi$ (equivalently, $\{\psi_i\}$) from $\mathcal{L}$ and $\sigma$ as one does classically; we construct $\mu_s$ as follows: $\sigma$ is a global section of $\mathcal{L}$ over $T\times X$ which vanishes at $\cup_i \psi_i^{-1}(\{s\}) = \psi^{-1}(D_s) \subset T$. So locally (in particular, over $\psi^{-1}(V_s)$) we may write

\[
\sigma|_{T \times \,\Spec\,\mathcal{O}_s} = \prod_{i=1}^n\psi_i^*(\varpi_s) e = \psi^*(\varpi_s^{(n)})e
\]

\noindent where $e$ is a generator (i.e., a nowhere-vanishing local section) of $\mathcal{L}|_{T \times  \,\Spec\,\mathcal{O}_s}$. Then, we may write

\[
\widetilde{\theta}_s = (\mu_s, e|_{T \times \{s\}}) \in \mathcal{M}_T\times^{\mathcal{O}_T^{\times}} \mathcal{L}^*|_{T \times \{s\}},
\]

\noindent where $\alpha(\mu_s) =  \prod_{i=1}^n\psi_i^*(\varpi_s) = \psi^*(\varpi_s^{(n)})$. Since such an $e$ is unique up to scaling by an element of $\mathcal{O}_{T \times \,\Spec\,\mathcal{O}_s}^{\times}$, we see that $\mu_s$ is actually a well-defined element of $\mathcal{M}_T(\psi^{-1}(V_s))$, giving us our desired section. 

\end{proof}

We observe that there are natural maps 

\begin{equation}\label{monoidalstr}
X^{(m)}_{\log}\times X^{(n)}_{\log} \to X^{(m+n)}_{\log},
\end{equation} 

\noindent given by $\left((\mathcal{L}, \sigma, \widetilde{\theta}_s), (\mathcal{L}', \sigma', \widetilde{\theta}_s')\right)\mapsto \left(\mathcal{L} \otimes\mathcal{L}', \sigma \otimes \sigma', \widetilde{\theta}_s\otimes \widetilde{\theta}_s'\right)$. This is analogous to the ``union" map $\Div^m(X) \times \Div^n(X) \to \Div^{m+n}(X)$.  More generally, if $n_1 + \cdots n_r = n$ is a partition of $n$, we have maps 

\begin{equation}\label{monoidalstrpart}
X^{(n_1)}_{\log}\times \cdots\times X^{(n_r)}_{\log} \to X^{(n)}_{\log}.
\end{equation} 

\noindent In particular, this endows $\sqcup_n X^{(n)}_{\log}$ with a monoidal structure. 

\subsection{The Map \texorpdfstring{$\psi_n: X^{(n)}_{\log } \to \overline{J}_{X,S}^n$}{psi n: X(n)toJXSn}}  Let us define the map of log spaces $\psi_n:X^{(n)}_{\log} \to \overline{J}_{X,S}^n$ by:

\[
\psi_n: (\mathcal{L}, \sigma, \widetilde{\theta}) \mapsto (\mathcal{L}, \theta)
\]

\noindent for every $(\mathcal{L}, \sigma, \widetilde{\theta}) \in X^{(n)}_{\log}(T)$. 

\begin{remark} Note that this is not quite a forgetful map: the $\widetilde{\theta}_s$ arising from $X^{(n)}_{\log}$ are identified under scaling by $\mathcal{O}_T^{\times}$ while the $\theta_s$ arising from $\overline{J}$ are identified under the more stringent scaling by $\mathcal{M}_T^{\gp}$. 
\end{remark}

\begin{remark} A slightly technical point is also worth mentioning here: the section $\sigma$ by assumption must be nonzero (and moreover $\sigma$ cannot vanish on any vertical fiber $\{t\}\times X \subset T \times X$). However, the restriction of $\sigma$ to the horizontal fibers $T\times\{s\}$, $s \in S$, may vanish. This occurs precisely when one of the maps $\psi_i: T \to X$ is constant onto $\{s\}$.
\end{remark}

\begin{remark}\label{AbelJacobi}
    Observe that for $n = 1$, we have the analogue of the Abel-Jacobi map: $\mathfrak{AJ}: X(\log S) \to \overline{J}_{X,S}^1$. This map extends the classical map $x \to (\mathcal{O}(x), 1, 1)$ for $x \in U = X \setminus S$. (Given a rational point $x_0 \in U$ we may of course tensor by $\mathcal{O}(-x_0)$ so that the Abel-Jacobi map lands in the logarithmic Jacobian $\overline{J}_{X,S}^0$, rather than the $\overline{J}^0$-torsor $\overline{J}^1$. But we shall adopt the slightly nonstandard formulation $\mathfrak{AJ}: X(\log S) \to \overline{J}_{X,S}^1\subset \overline{\Pic}_{X,S}$.)
\end{remark} 

\begin{prop}
    $\psi_n: X^{(n)}_{\log } \to \overline{J}_{X,S}^n$ is surjective for all $n \ge 2g -1 + |S|$.
\end{prop}

\begin{proof}
    Firstly, we observe that given a collection $\theta_s$ of global sections of $\mathcal{M}_T^{\gp} \times^{\mathcal{O}_T^{\times}} (\mathcal{L}^{*}|_{T \times\{s\}})$, we may ``clear denominators" in $\mathcal{M}_T^{\gp}$ to give an equivalent collection $\theta_s \in \mathcal{M}_T \times^{\mathcal{O}_T^{\times}} (\mathcal{L}^{*}|_{T \times\{s\}})$. We may then consider $\alpha(\theta_s)$ for each $s$, and these give us sections $a_s \in \mathcal{L}^{*}|_{T \times \{s\}}$. Now we may proceed as in the classical story: we have the exact sequence

\begin{equation}\label{exactseq}
\begin{tikzcd}
0 \arrow[r] & {H^0(X_T, \mathcal{L}(-S))} \arrow[r] & {H^0(X_T, \mathcal{L})} \arrow[r] & {H^0(X_T, \mathcal{L}/\mathcal{L}(-S))} \arrow[r] & {H^1(X_T, \mathcal{L}(-S))}
\end{tikzcd}
\end{equation}

\noindent where $X_T = T \times X$. But since the degree of $\mathcal{L}(-S)$ is greater than $2g-1$, Riemann-Roch shows that $H^1(X, \mathcal{L}(-S)) = 0$. Hence the map $H^0(X_T, \mathcal{L}(-S)) \to H^0(X_T, \mathcal{L})$ is surjective, and so, in particular, $(a_s)_{s \in S} \in H^0(X_T, \mathcal{L}/\mathcal{L}(-S))$ lies in the image of a global section $\sigma$ of $\mathcal{L}$.

\end{proof}

We now study the fibers of the map $X^{(n)}_{\log } \to \overline{J}_{X,S}^n$. Clearly, given $(\mathcal{L}, \theta) \in \overline{J}_{X,S}^{n} (T)$, the collection of preimages is given by:

\[
\psi_n^{-1}(\mathcal{L}, \theta := (\theta_s)_{s \in S}) = \{(\mathcal{L}, \sigma, \widetilde{\theta}:=(\widetilde{\theta}_s)_{s\in S})\}/\sim,
\]

\noindent where:

i)  $\sigma \in H^0(X_T, \mathcal{L})$ a global section of $\mathcal{L}$, such that the image of $\sigma$ in  $H^0(X_T, \mathcal{L}/\mathcal{L}(-S)) = \mathcal{L}|_S$ under (\ref{exactseq}) agrees with $\alpha(\theta_s)$;

ii) $\widetilde{\theta} = \{\widetilde{\theta}_s\}_{s \in S}$ is a lift of $\theta$, such that $\widetilde{\theta}_s \in H^0(T\times \{s\}, \mathcal{M}_T \times^{\mathcal{O}_T^{\times}}\mathcal{L}^{*}|_{T \times \{s\}})$;

\noindent where 

iii) we identify $(\sigma, \theta) \sim (\sigma', \theta')$ if they differ by an overall $\Gamma(\mathcal{O}_{T}^{\times})$ scalar factor. (Observe that since $X$ is projective,  $\Gamma(\mathcal{O}_{T\times X}^{\times}) \simeq \Gamma(\mathcal{O}_T^{\times})$.) 

Observe that the space of such $\sigma$  -- i.e., those with specified image in $H^0(X_T, \mathcal{L}/\mathcal{L}(-S)) = \mathcal{L}|_S$ (and not considering the equivalence relation in iii)) -- is an affine linear subspace of $H^0(X_T, \mathcal{L})$. I.e., it is an $H^0(X_T, \mathcal{L}(-S))$-torsor, by (\ref{exactseq}).

Zariski-locally on $T$, we may assume that $\mathcal{L}|_{T \times S}$ is trivial, so we may consider $(\theta_s)_{s \in S}$ to lie in $(M_T^{\gp})^{|S|}$. Then the space of lifts $\{\widetilde{\theta_s}\}_{s \in S}$ (again, forgetting the equivalence relation) is the preimage of $\{\theta_s\}_{s\in S}$ under 

\[
M_T^{|S|} \longrightarrow ( M_T^{\gp})^{|S|}/M_T^{\gp}
\]

\noindent where $M_T^{\gp}$ acts by uniform scaling of $(M_T^{\gp})^{|S|}$. Equivalently, this space is given by:

\begin{equation}\label{fiberlift}
\{(M_T^{\gp}\cdot (\theta_s)_{s \in S} \cap M_T^{|S|}) \}.
\end{equation}

Observe that there may not be a greatest common divisor of the $\theta_s$ in $M_T^{\gp}$. Thus we cannot simply write (\ref{fiberlift}) as $\{t \cdot \theta^0\}$, where $\widetilde{\theta}^0 \in M_T^{|S|}$ is some fixed generator of the ``line" (\ref{fiberlift}). However, if $T$ is a \textit{valuative} local log scheme, then this problem goes away.

Recall that an affine fs log scheme $(T, \mathcal{O}_T, M_T)$ is \textit{valuative} if for any $t \in T$ the stalk $\mathcal{M}_t$ satisfies $\mathcal{M}_t^{\gp} = \mathcal{M}_t \cup \mathcal{M}_t^{-1}$ (\cite{KatoVal}). In such a case, the $\widetilde{\theta}_s$ in any lift of $\theta$ are totally ordered by divisibility, and so, by scaling by the inverse of a minimal $\widetilde{\theta}_s$, we may find a representative $\widetilde{\theta}^0$ such that any minimal element is a unit in $M_{T}$. Thus for a valuative affine test log scheme $T$, we find that (\ref{fiberlift}) is $\{M_T \cdot \widetilde{\theta}^0\}$; in other words, (\ref{fiberlift}) is identified with $M_T$.

So now, let $(\mathcal{L}, \theta) \in \overline{\Pic}_{X,S}^n(T)$, where (restricting if necessary) we may assume $T$ is a valuative, fs, affine, test log scheme over which $\mathcal{L}|_{T\times\{s\}}$ is trivial for all $s \in S$. We assume that $n \ge 2g-1 + |S|$. Let $V := H^{0}(X, \mathcal{L})$, which is a finite-dimensional $k$-vector space. The map $H^{0}(T \times X, \mathcal{L}) \to H^{0}(T \times X, \mathcal{L}/\mathcal{L}(-S))$ will be denoted by $ev_S$, and, after trivializing $\mathcal{L}|_{T\times \{s\}} \simeq \mathcal{O}_T$ for all $s\in S$, this can be written as $ev_S: V \otimes \mathcal{O}_T \to \mathcal{O}_T^{|S|}$. Pick an $\mathcal{M}_T$-generator $\widetilde{\theta}^0$ lifting $\theta$ as above. Then the fiber $\psi_n^{-1}( (\mathcal{L}, \theta))$ is given by global sections of the sheaf (of sets):

\begin{equation}\label{fiberT}
\left( [V \otimes \mathcal{O}_T]^\times \times_{\mathcal{O}_T^{|S|}} \mathcal{M}_T\right)/\mathcal{O}_T^{\times},
\end{equation}

\noindent where we are adopting notation from Proposition \ref{CompactVectPt}.

We observe that \ref{fiberT} represents sections of a family of compactified vector spaces over $T$. Indeed, for $t \in T$ a geometric point, we notice that the fiber over $t$ is given by a point of the log scheme $\P(V_t)$ with divisorial log structure at the hyperplane $\P(H_t)$ -- where 

i) $V_t$ is the linear subspace of $V$ given by the preimage $ev_S^{-1}(L)$, where $L$ is the line generated by $\alpha(m_t) \widetilde{\theta}^0$ (with $m_t$ the restriction of $m \in \mathcal{M}_T$ to $t \in T$), and 

ii) $H_t  := ev^{-1}(0) \subseteq V_t$.

Applying Proposition \ref{CompactVectPt} and the notion of a relative compactified vector space, we have:

\begin{thm}\label{mainthm}
    Say $|S|>1$, and $n \ge 2g-1 + |S|$. The map $\psi_n: X^{(n)}_{\log} \to \overline{J}_{X,S}^{n}$ is surjective; moreover, it is valuative-locally a compactified vector bundle, with fibers of dimension $\dim H^0(X, \mathcal{L}) - |S|$ for (any) $\mathcal{L} \in \Pic_X^n(k)$.
\end{thm}

\begin{remark}\label{htpyequivalencemap}
    In fact, given the fact that compactified vector spaces are \textit{contractible} (see Remark \ref{contractibility}), we can see that $X^{(n)}_{\log}$ has the same homotopy type as $J_X^{(n)}$ for sufficiently large $n$. This is a \textit{stronger} claim than can be made in Deligne's classical unramified geometric class field theory, where the equivalence is only of $\pi_1$'s. (Indeed, in classical unramified geometric class field theory the fiber of $X^{(n)} \to J_{X}^{n}$ is a projective space, which has nontrivial homotopy.)
\end{remark}

\subsection{Example: \texorpdfstring{$\P^1$ with Log at 0 and $\infty$}{P1 with Log at 0 and infty}}  In this subsection we will examine a concrete example: $\P^1$ with two marked points. Specifically, consider $\P^1$ with a log structure at 0 and $\infty$. As a scheme, the symmetric power $(\P^1)^{(n)} = \P^N$, and the log structure becomes the log structure of two hyperplanes; in homogeneous coordinates, we have the map 

\[
(\P^1)^{(n)} \to \P^n: \{[x_0: y_0], \ldots, [x_n: y_n]\} \mapsto [a_0:\ldots:a_n]
\]

\noindent where the $a_i$ are given by the coefficients of

\[
f(X,Y) := \prod_{i=0}^{n} (Yx_i - Xy_i) = \sum_{i=0}^{n} a_i X^i Y^{n-i}
\]

\noindent written as symmetric homogeneous polynomials of $x_i$ and $y_i$. The hyperplane sections are then given by $a_0 = 0$ (corresponding to $S^{(0)}$) and $a_n = 0$ (corresponding to $S^{(\infty)}$). Thus locally on an affine test log scheme $T$, we find that $(\P^1)^{(n)}_{\log}(T)$ is given by:

\[
\{(\mathcal{M}_T \times \mathcal{O}_T\times \cdots\times \mathcal{O}_T\times \mathcal{M}_T): \text{these $n+1$ factors do not simultaneously vanish}\}/\mathcal{O}_T^{\times},
\]

\noindent where we mean that the $n-1$ middle factors of $\mathcal{O}_T$, along with $\alpha$ of the $\mathcal{M}_T$ factors on the flanks, must generate the unit ideal in $\mathcal{O}_T$.

On the other hand, the Jacobian of $\P^1$ is trivial, so that:

\[
\overline{J}_{X,S}^n = \mathbb{G}_m^{\log}.
\]

\noindent After trivializing the copy of $\mathcal{M}_T$ at infinity, we can represent the map $\psi_n$ by 

\[
(m_0, f_1, \ldots, f_{n-1}, m_n) \mapsto m_0m_n^{-1} \in \G_m^{\log}(T).
\]

\noindent We see that the fiber is $\{(m_0, f_1, \ldots f_{n-1}, m_n)\}/\mathcal{O}_T^{\times}$ with fixed $m_0 m_n^{-1}$, which, for valuative $T$, is given by 

\[
\{(\mathcal{O}_T\times \cdots\times \mathcal{O}_T\times\mathcal{M}_T): \text{these $n$ factors do not simultaneously vanish}\}/\mathcal{O}_T^{\times};
\]

\noindent this is a compactified vector space.

\begin{remark}
    Observe that in the case of $n=1$, we have the homotopy equivalence $\P^1_{\log} \to \G_m^{\log}$ discussed in Section \ref{htpylogstack}. This is a special case of the general claim that $X^{(n)}_{\log}$ and $\overline{J}_{X, S}^{n}$ have the same homotopy type (see Remark \ref{htpyequivalencemap}).
\end{remark}

\section{Logarithmic Geometric Class Field Theory}\label{CFTLogGeometry}

\subsection{The Fundamental Correspondence} Theorem \ref{mainthm} now provides us with the following.

\begin{thm}\label{LogGCFT}
    Let $\overline{x}$ be a geometric point of $X$ lying over a rational point $x \in X(k)$. There is a canonical equivalence between the categories of
    
    1) 1-dimensional $\ell$-adic local systems $\mathfrak{L}$ on $X(\log S)$, with a fixed isomorphism $\mathfrak{L}_{\overline{x}} \simeq \overline{\mathbb{Q}_{\ell}}$; and 
    
    2) multiplicative 1-dimensional $\ell$-adic local systems (character sheaves) on $\overline{\Pic}_{X,S}$.
\end{thm}

\begin{proof}
   
    We essentially follow Deligne's argument in the classical case. Given an $\ell$-adic local system $\mathfrak{L}$ on $\overline{\Pic}_{X,S}$, we pull back to $X(\log S)$ under the Abel-Jacobi map $\mathfrak{AJ}$ (Remark \ref{AbelJacobi}). This provides one direction. For the other, consider a local system $\mathfrak{L}$ on $X(\log S)$. We may construct the local system $\mathfrak{L}^{\boxtimes n}:=\mathfrak{L} \boxtimes \cdots \boxtimes \mathfrak{L}$ on the $n$th power $X(\log S)^n$ for each $n \ge 0$. Recall that we have a map $p_n: X(\log S)^n \to X^{(n)}_{\log}$ (\ref{monoidalstrpart}). We define

    \begin{equation}
        \mathfrak{L}^{(n)}: = \left(p_{n*}\mathfrak{L}^{\boxtimes n}\right)^{S_n},
    \end{equation}

    \noindent giving a log-\'etale sheaf $\mathfrak{L}^{(n)}$ on $X^{(n)}_{\log}$ for each $n \ge 0$. Here we take $S_n$-invariants, with $S_n$ acting on $X(\log S)^n$, and so on $p_{n*}\mathfrak{L}^{\boxtimes n}$, via permutation of coordinates. 
    
    Importantly, $\mathfrak{L}^{(n)}$ is a local system; the argument establishing this is essentially similar to the argument presented in \cite{toth2011geometric} in the classical setting. However, this is a slightly subtle point and involves proper base-change and representability of sheaves in the Kummer-\'etale setting \cite{NakLogEtI}, so we will review it here.
    
    Consider a log geometric point $\overline{x} = (x_1, \ldots, x_n) \in X(\log S)^n$. Let $i: p_n(\overline{x}) \to X^{(n)}_{\log}$, and consider the Cartesian square:

    \[
\begin{tikzcd}
p_n^{-1}p_n(\overline{x}) \arrow[d, "p_n'"] \arrow[r, "i'"] & X(\log S)^n \arrow[d, "p_n"] \\
p_n(\overline{x}) \arrow[r, "i"]                            & X^{(n)}_{\log}.              
\end{tikzcd}
    \]

    \noindent The map $X(\log S)^n \to X^{(n)}_{\log}$ is proper (which, as we have discussed, is only true in our logarithmic setting, and \textit{not} true for the classical tamely ramified $U^{(n)}\to\Pic_{X,S}^n$), so we may apply proper base change \cite{NakLogEtI}, which says that $i^*p_{n*} \mathfrak{L}^{\boxtimes n} \to p'_{n*}(i')^*(\mathfrak{L})$ is an isomorphism. Thus 
    
    \[\left(p_{n*} \mathfrak{L}^{\boxtimes n}\right)_{p_n(\overline{x})} \simeq \prod_{\overline{y} \in X(\log S)^n\,:\, p_n(\overline{y}) = \overline{x}}(i')^*(\mathfrak{L})_{\overline{y}}\simeq \prod_{\sigma \in S_n} \left(\bigotimes_i\mathfrak{L}_{x_{\sigma(i)}}\right).\]
    
    \noindent There is a natural isomorphism $\bigotimes_i\mathfrak{L}_{x_{i}} \simeq \prod_{\sigma \in S_n} \left(\bigotimes_i\mathfrak{L}_{x_{\sigma(i)}}\right)^{S_n}$ given by the symmetric embedding, so we see that $\left(\left(p_{n*} \mathfrak{L}^{\boxtimes n}\right)^{S_n}\right)_{p_n(\overline{x})} \simeq \bigotimes_i\mathfrak{L}_{x_{i}}$. Thus $\left((p_n')^*\left(\left(p_{n*} \mathfrak{L}^{\boxtimes n}\right)^{S_n}\right)\right)_{\overline{x}} \simeq \bigotimes_i\mathfrak{L}_{x_{i}} \simeq (\mathfrak{L}^{\boxtimes n})_{\overline{x}}$. These computations let us conclude from stalks that 
    
    \[(p_n')^*\left(\mathfrak{L}^{(n)}\right) = (p_n')^*\left(\left(p_{n*} \mathfrak{L}^{\boxtimes n}\right)^{S_n}\right) \simeq \mathfrak{L}^{\boxtimes n}.\]
    
    So far we have shown that the sheaf $\mathfrak{L}^{(n)}$ has 1-dimensional stalks and that its pullback to the product $X(\log S)^n$ is $\mathfrak{L}^{\boxtimes n}$, which is a Kummer local system. We must show that $\mathfrak{L}^{(n)}$ is a Kummer local system. For a moment, consider finite ($\Z/\ell^n\Z$)-coefficients rather than $\overline{\Q}_{\ell}$ coefficients. We note that since $\mathfrak{L}^{\boxtimes n}$ is a local system, it is representable by some log scheme $Z'$ which is Kummer log-\'etale over $X(\log(S))^n$. Since pushforwards and subsheaves of representable sheaves are representable, we have some log scheme $Z$, Kummer over $X^{(n)}_{\log}$, which will represent $\mathfrak{L}^{(n)}$. 
    
    As above, let $\overline{x}$ be a geometric point of $X(\log(S))^n$ and $\overline{y} = p_n(\overline{x}) \in X^{(n)}_{\log}$ the image of the point ``downstairs". By representability, we will have a Kummer \'etale neighborhood $U$ of $\overline{y}$ and a map $U \to Z$ generating the (1-dimensional) stalk at $\overline{y}$. Pulling this back to $X(\log(S))^n$, we find that $p_n^{-1}(U) \to  Z\times_{X^{(n)}_{\log}} X(\log S)^n$ generates the (one-dimensional) stalk of $\mathfrak{L}^{\boxtimes n}$ at $\overline{x}$. Thus the pullback generates all the stalks in the neighborhood $p_n^{-1}(U)$, which implies that $U \to Z$ generates all the stalks of $\mathfrak{L}^{(n)}$ in a Kummer-\'etale neighborhood of $\overline{y}$. Hence $\mathfrak{L}^{(n)}$ (with finite coefficients) is locally constant, and we may pass to the inverse system of these coefficients, and invert $\ell$, to show that $\mathfrak{L}^{(n)}$ is a locally constant $\overline{\Q}_\ell$-adic sheaf on the Kummer-\'etale site of $X^{(n)}_{\log}$.
    
    By Theorem \ref{mainthm} and \ref{simpleconnfiber}, the map $X^{(n)}_{\log} \to \overline{J}_{X,S}^{n}$ induces an isomorphism of $\pi_1^{\log}$ for sufficiently large $n$ (in particular, for $n \ge N := 2g-1+|S|$). Moreover, the monoidal structure on $\sqcup_nX^n_{\log}$ (\ref{monoidalstr}) clearly descends to multiplication on $\overline{\Pic}_{X,S}= \sqcup_n \overline{J}_{X,S}^n$. Transferring these local systems to $\overline{J}_{X,S}^{n}$ for sufficiently large $n$, we obtain a local system on $\overline{\Pic}_{X,S}^{\ge N}$ with the multiplicative property of 2).
    
    We then may use this multiplicativity along with the fixed rigidification over $\overline{x}$ to extend the sheaf to all of $\overline{\Pic}_{X,S}$. This, along with the fact that these two constructions are inverses, is a verbatim repeat of Deligne's argument in the unramified case (\cite{toth2011geometric}, Section 2.2). 
\end{proof}

\subsection{The Tamely Ramified Artin Map} Let us say that $k = \F_q$ is a finite field. Theorem \ref{LogGCFT} permits us to immediately re-derive the familiar results for tamely ramified global class field theory of a function field in one variable over $\F_q$. 

Firstly, we observe that, giving $\F_q$ trivial log structure, $X(\log S)(\F_q) = U(\F_q)$; $X^{(n)}_{\log}(\F_q) = U^{(n)}(\F_q)$; and  $\overline{\Pic}_{X,S}(\F_q) = \Pic_{X,S}(\F_q)$ (with $\Pic_{X,S}$ the classical rigidified Picard; see Section \ref{ClassicalRigPic}). This reduces us to the case of classical tamely ramified geometric class field theory, discussed in detail in \cite{toth2011geometric}, Section 4.2. We will review this story, leaving some technical details to loc. cit..

Recall that if $p \in U(\F_q)$, then choosing a geometric point $\overline{p}$ above $p$, we have a map $\pi_1^{et}(\Spec(\F_q), \overline{p}) \to  \pi_1^{et}(U, \overline{p})$; we recall that $\pi_1^{et}(\Spec(\F_q), \overline{p})$ is $\widehat{\Z} = \Gal(\overline{\F_q}/\F_q)$, with a generator given by Frobenius, $\Frob_p$. We may abusively also refer to the image of $\Frob_p$ in $\pi_1^{et}(U, \overline{p})$ as $\Frob_p$. Taking $\pi_1^{et}(U,\overline{x})$ with respect to a different geometric point $\overline{x}$ results in a non-canonically isomorphic $\pi_1$; however, the two abelianizations of these groups \textit{are} canonically isomorphic. We can therefore write the abelianization of $\pi_1$ without reference to a basepoint; e.g., $\pi_1^{et,\textrm{ab}}(U)$. We write $[\Frob_p]$ for the image of $\Frob_p$ in this group.

We have a ``reciprocity" map $\sqcup_n X^{(n)}_{\log}(\F_q) = \sqcup_n U^{(n)}(\F_q) \to \pi_1^{t,\textrm{ab}}(U) =  \pi_1^{\log, \textrm{ab}}(X(\log S))$ given by 

\begin{equation}\label{RecMap}
\textrm{Rec}: \{p_1, \ldots , p_n\} \mapsto \prod_{i=1}^n [\Frob_{p_i}].
\end{equation}

\noindent (Note that we permit the collection $\{p_1, \ldots p_n\}$ to be a multi-set, so that we may have powers of $[\Frob_p]$ on the right-hand side of \ref{RecMap}.) 

We claim that the reciprocity map \ref{RecMap} descends to $\overline{\Pic}_{X,S}(\F_q) = \Pic_{X,S}(\F_q)$.  To see this, consider a character $\chi: \pi_1^{\log, \ab}(X(\log S)) \to \overline{\Q_{\ell}^{\times}}$. Given a geometric point $x$, we see that this induces a character of $\pi_1^{\log}(X(\log S), \overline{x})$ via precomposing with the surjection $\pi_1^{\log}(X(\log S), \overline{x}) \to \pi_1^{\log, \textrm{ab}}(X(\log S), \overline{x})$. This in turn corresponds to a 1-dimensional local system on $X(\log S)$ (with rigidification data at $\overline{x}$). 

Now we may apply our Theorem \ref{LogGCFT}; from this 1-dimensional local system on $X(\log S)$ we obtain a canonically-associated multiplicative local system on $\overline{\Pic}_{X,S}$. By Proposition \ref{CharsheavesGmlog}, this multiplicative local system corresponds to a character $\xi_{\chi}: \overline{\Pic}_{X,S}(\F_q) \to \overline{\Q}_{\ell}^{\times}$. A routine verification (\cite{toth2011geometric}, 4.2.3) shows that for all $\chi$ the following diagram commutes:

\begin{equation}\label{RecClassDiagram}
\begin{tikzcd}
\sqcup_n X^{(n)}_{\textrm{log}}(\mathbb{F}_q) \arrow[r, "\textrm{Rec}"] \arrow[d, "\psi"'] & {\pi_1^{\textrm{log}, \textrm{ab}}(X(\log S))} \arrow[d, "\chi"] \\
{\overline{\textrm{Pic}}_{X,S}(\mathbb{F}_q)} \arrow[r, "\xi_\chi"']                       & \overline{\mathbb{Q}_{\ell}}^{\times}.                           
\end{tikzcd}
\end{equation}

\noindent Because we may freely choose the character $\chi$, we deduce that $\textrm{Rec}$ factors through $\psi$ -- in other words, $\textrm{Rec}(D)$ depends only on the rational equivalence class of a divisor $D$. Thus we obtain the tamely ramified global Artin map 

\begin{equation}
\textrm{Artin}: \overline{\Pic}_{X,S}(\F_q) \to \pi_1^{\textrm{log}, \textrm{ab}}(X(\log S)).
\end{equation}

\noindent Moreover, we can see that there is a bijection between finite-index subgroups of $\overline{\Pic}_{X,S}(\F_q)$ and $\pi_1^{\textrm{log}, \textrm{ab}}(X(\log S))$.

Translating everything into ad\`elic language, we let $K$ be the fraction field of $X$, and $\mathbb{I}_K$ the associated group of id\`eles. Set $\mathbb{O}_{K,S} := \prod_{p \in U(\F_q)} \widehat{\mathcal{O}_{p}}^{\times}\times \prod_{p \in S} \widehat{\mathcal{O}_{p}}^{\times(1)}$, where $\widehat{\mathcal{O}_{p}}$ is the 
formal completion of the stalk $\mathcal{O}_{p}$ at $p$, $\widehat{\mathcal{O}_{p}}^{\times}$ is the group of units of this formal completion, and $\widehat{\mathcal{O}_{p}}^{\times(1)} = \{x \in \widehat{\mathcal{O}_{p}}^{\times}: x \equiv 1 \,\,\textrm{mod}\,\,\mathfrak{m}_p\}$. Additionally, set $\mathbb{I}_K^{(S)} := \prod'_{p \notin S}\widehat{K_p}$, set $\mathbb{O}_{K}^{(S)} := \prod_{p \in U(\F_q)} \widehat{\mathcal{O}_{p}}^{\times}$, and let $K_S^{\times} = \{f \in K(X) : f(p) = 1 \textrm{ for all }p \in S\}$.

Recall the Weil uniformization (\cite{MilneCFT}, Proposition 4.6): 

\begin{equation}\label{WeilUniform}
\overline{\Pic}_{X,S}(\F_q) \simeq K^{\times} \backslash \mathbb{I}_K/\mathbb{O}_{K,S} \simeq K^{\times}_S \backslash \mathbb{I}_K^{(S)}/\mathbb{O}_{K}^{(S)}.
\end{equation}

Let $\overline{K^{t,S}}$ be the maximal abelian Galois extension of $K$ that is tamely ramified over $S$. We observe, combining Theorem \ref{FundIso} with the correspondence between \'etale $\pi_1$ and Galois groups, that $\pi_1^{\textrm{log}, \textrm{ab}}(X(\log S))\simeq\Gal\left(K^{t,S,\ab}/K\right)$. We define a homomorphism, also somewhat abusively called $\textrm{Rec}: \mathbb{I}_K^{(S)}/\mathbb{O}_K^{(S)} \to \Gal\left(K^{t,S,\ab}/K\right)$, given by:

\[(a_p)_p \mapsto \prod_p [\Frob_p]^{v_p(a_p)}.\]

\noindent This manifestly factors through $\mathbb{O}_K^{(S)}$; it corresponds, at least when $v_p(a_p) \ge 0$ for all $p \in U(\F_q)$, to the geometric map (\ref{RecMap}). The fact that the geometric $\textrm{Rec}$ (\ref{RecMap})  factors through $\overline{\Pic}_{X,S}(\F_q)$ (which, crucially, is what required the full power of Theorem \ref{LogGCFT} in our argument above) is tantamount to the claim that the id\`elic $\textrm{Rec}$ is trivial on global rational functions $K^{\times}_S \subset \mathbb{I}_K^{(S)}$.

Bringing everything together, we see that finite-index subgroups of $K^{\times} \backslash \mathbb{I}_K/\mathbb{O}_{K,S}$ correspond to finite index subgroups of $\Gal\left(K^{t,S,\ab}/K\right)$; in other words, there is a canonical bijection between finite-index subgroups of $K^{\times} \backslash \mathbb{I}_K/\mathbb{O}_{K,S}$ and abelian extensions of $K$, tamely ramified over $S$. We summarize this as follows:

\begin{thm}\label{ClassClassfieldtheory}
    There is a map $\textrm{Artin}: \overline{\Pic}_{X,S}(\F_q) \to \pi_1^{\textrm{log}, \textrm{ab}}(X(\log S))$, that corresponds to the tamely ramified global Artin map under the Weil uniformization $\overline{\Pic}_{X,S}(\F_q) \simeq K^{\times} \backslash \mathbb{I}_K/\mathbb{O}_{K,S} \simeq K^{\times}_S \backslash \mathbb{I}_K^{(S)}/\mathbb{O}_{K}^{(S)}$, and under the isomorphism $\pi_1^{\log\textrm{ab}}(X(\log S)) \simeq \Gal\left(K^{t,S,\ab}/K\right)$. This map induces isomorphisms on finite quotients, showing that abelian extensions of $K$ tamely ramified over $S$ correspond to finite-index subgroups of $\Pic_{X,S}(\F_q)$.
\end{thm}

\subsection{Logarithmic Structures and Local-to-Global Compatibility}\label{locglobcomp} In \ref{RecClassDiagram}, we considered only $\F_q$-points, which effectively ignores the logarithmic structures which are present, and brings us back to the analysis of the classical $\Pic_{X,S}$. (It is worth noting, however, that the logarithmic upgrade improved the proofs of Theorem \ref{mainthm} and Theorem \ref{LogGCFT} by bypassing ad hoc compactifications.) We shall presently show how the logarithmic structures permit us to understand local-to-global behavior on the ramified locus $S$.

Firstly, we observe the following, which can be viewed as logarithmic tamely ramified local class field theory:

\begin{prop}\label{LocLogCFT} (Local Logarithmic Class Field Theory.)
    There is a canonical map $\textrm{Rec}_{\textrm{loc}}: \G_m^{\log}(\F_q^{\log}) \to \pi_1^{\log \ab}(\F_q^{\log})$. The latter group is the profinite completion of the former.
\end{prop}

\noindent Observe that this is analogous to the local reciprocity map $\G_m(F) = F^{\times} \to \Gal(F^{\ab}/F)$ for $F$ a local field with residue field $\F_q$. At the tamely ramified level this corresponds to $F^{\times}/\mathcal{O}_F^{\times (1)}\to \Gal(F^{t,\ab}/F)$ (with $F^{t,\ab}$ the maximal tamely ramified extension of $F$), which recognizably correspond to the two sides of Proposition \ref{LocLogCFT}.

If we write $\Spec(\mathcal{O}_F) = D$ and $\Spec(F) = D^*$, where $D$ is the formal disk and $D^*$ is the punctured formal disk, then we may endow $D$ with the standard divisorial log structure (coming from $\mathfrak{m}_F\subset \mathcal{O}_F$), which we may view as $D^*$ with a logarithmic boundary (visually: the puncture becomes a closed circle). 
Then the inclusion of the closed fiber $\F_q^{\log} \to D^{\log}$ induces an isomorphism on fundamental groups (since both are the tamely ramified fundamental group by \ref{FundIso}). (In fact, this is a homotopy equivalence between $\F_q^{\log}$ and $D^{\log}$, though we will not further elaborate on this point.) 
We also notice in passing that everything in this discussion holds true for arithmetic local fields like $\Q_p$ as well as rings of formal series over $\F_q$, though for our purposes we will only be interested in the latter.

\begin{proof}
    Since the two groups are given by $\F_q^{\times}\times \Z$ and $\F_q^{\times} \times \widehat{\Z}$, we may simply note that there is an obvious inclusion from one to the other, which makes the latter the profinite completion of the former. However, it is important that we are careful to show that the correspondence is in fact \textit{canonical}. To accomplish this, we explicitly construct the map in question. 
    
    Firstly, $\G_m^{\log}(\F_q^{\log})$ is canonically $\F_q^{\times} \oplus \Z$. Next, we shall write down $\pi_1^{\log}$ and $\pi_1^{\log, ab}$ in terms of automorphism groups of specific Kummer covers. 
    Observe that an analogue of the ``universal cover" of $\F_q^{\log}$ is $\overline{\F_q}^{Kum}$, defined as follows. The underlying scheme is $\Spec(\overline{\F_q})$, and the log structure is that associated to the monoid $\N_{(p)} = \{a/n\in \Q : (a,n) = 1, p \nmid n, a/n \ge 0\}$ with $\alpha$ sending $\N_{(p)}$ to 0. The automorphisms of the cover $ \overline{\F_q}^{Kum}\to \F_q^{\log}$ give the full $\pi_1^{\log}$ of $\F_q^{\log}$, which can be easily computed to be $\widehat{\Z(1)}'\rtimes_q \widehat{\Z}$. On the other hand, the universal \textit{abelian} cover of $\F_q^{\log}$ is given by $\overline{\F_q^{\log, ab}}$, defined as follows: the underlying scheme is $\Spec(\overline{\F_q})$, with log structure associated to the monoid $\frac{1}{q-1}\N$ (with $\alpha$ sending $\frac{1}{q-1}\N$ to 0). In other words, $M_{\overline{\F_q^{\log, ab}}}= \overline{\F_q}^{\times} \oplus \frac{1}{q-1}\N$.
    
    The cover $\overline{\F_q^{\log, ab}}\to\F_q^{\log}$ has Galois group canonically isomorphic to $\F_q^{\times} \times \widehat{\Z}$, given as follows: for $(u, t) \in \F_q^{\times} \times \widehat{\Z}$, and $\left(x, \frac{m}{q-1}\right) \in \overline{\F_q}^{\times} \oplus \frac{1}{q-1}\N$, we have:

    \[
    (u, t): \left(x, \frac{m}{q-1}\right) \mapsto \left(u^mx^{q^t}, \frac{m}{q-1}\right).
    \]

    \noindent Observe that $u^q = u$ since $u \in \F_q$, and that if $(q-1)| m$ then $u^m = 1$ so this action is well-defined and preserves the inclusion $\F_q^{\times} \oplus \N \subset \overline{\F_q}^{\times} \oplus \frac{1}{q-1}\N$.

    The map $\F_q^{\times} \oplus \Z \to \F_q^{\times} \oplus \widehat{\Z}$ is now canonically associated with $\G_m^{\log}(\F_q^{\log}) \to \pi_1^{\log \ab}(\F_q^{\log})$.
\end{proof}

Consider now $X(\log S)$. We have an obvious map $X(\log S) \to X$, and the fiber over $s \in S \subseteq X(\F_q)$ is a standard log point $\F_q^{\log}$. We shall write this $\F_q^{\log}$-point of $X(\log S)$ as $\log(s)$, and we shall write the fiber of $X(\log S)$ over the formal disk $D_s$ of $s$ as $D_s^{\log}$. 

There is a natural group homomorphism:

\begin{equation}\label{imap}
   i_s: \mathbb{G}_m^{\log}(\log(s)) \to \overline{\Pic}_{X,S}(\F_q).
\end{equation} 

\noindent This is defined as follows: let $m \in \G_m^{\log}(\log(s))$. From this we must produce an $\F_q$-point of $\overline{\Pic}_{X,S}$; that is, a line bundle $\mathcal{L}$ with a section over each fiber $i_{x}^*(\mathcal{L})$ for $x \in S$. Given $m$, we may lift this to an element of $\G_m^{\log}(D_s^{\log})$, which in turn corresponds to a nonzero rational function on $D$; i.e., a nonzero function $f$ on the formal punctured disk $D^*$.

To construct $\mathcal{L}$ we apply Beauville-Laszlo gluing along the formal punctured disk $D_s^{\times}$ \cite{BeauvilleLaszlo1995}. Take a trivial bundle on $X \setminus \{s\}$, restrict to $D_s^*$, and then glue to the trivial bundle over $D$ via $f: \textrm{Triv}_{X\setminus\{s\}}|_{D_s^*} \to \textrm{Triv}|_D$. The section corresponding to 1 on the special fiber of $D$ and $1 \in \Gamma(\textrm{Triv})$ for all $s' \ne s$ in $S$, gives the  corresponding rigidification data. Observe that this is well-defined since $f$ is unique up to multiplication by $\widehat{\mathcal{O}_s}^{\times (1)}$.

Finally, we note that $\log(s)$ corresponds to an $\F_q^{\log}$-point of $X(\log S)$, the inclusion $\log(s) \to X(\log S)$ induces a map of abelianized fundamental groups:

\begin{equation}\label{s*map}
    s_*:\pi_1^{\log \ab}(\F_q^{\log}) \to \pi_1^{\log \ab}(X(\log S)).
\end{equation}

We now put everything together. The subsequent theorem follows from our constructions:

\begin{thm}\label{loctoglo} (Logarithmic Local-to-global compatibility.) The following diagram commutes:

\begin{equation}
\begin{tikzcd}
\mathbb{G}_m^{\log}(\mathbb{F}_q^{\log}) \arrow[r, "\textrm{Rec}_{\textrm{loc}}"] \arrow[d, "i_s"'] & \pi_1^{\log \textrm{ab}}(\mathbb{F}_q^{\log}) \arrow[d, "s_*"] \\
{\overline{\textrm{Pic}}_{X,S}(\mathbb{F}_q)} \arrow[r, "\textrm{Artin}"]                           & \pi_1^{\log \textrm{ab}}(X(\log S)),                           
\end{tikzcd}
\end{equation}

\noindent and this corresponds to the map

\begin{equation}
\begin{tikzcd}
\widehat{K_s}^{\times}/\widehat{\mathcal{O}_s}^{\times (1)} \arrow[r, "\textrm{Rec}_{\textrm{loc}}"] \arrow[d] & {\textrm{Gal}\left(\overline{K_s^{t,\textrm{ab}}}/K_s\right)} \arrow[d] \\
{K^{\times} \backslash\mathbb{I}_K/\mathbb{O}_{K,S}} \arrow[r, "\textrm{Artin}"]                               & {\textrm{Gal}\left(\overline{K^{t,\textrm{ab}}}/K\right)}              
\end{tikzcd}
\end{equation}

\noindent from class field theory.

\end{thm}

Thus we see that the log structures at the ramified points give a geometric interpretation of each factor in the full id\`elic product $K^{\times} \backslash \mathbb{I}_K/\mathbb{O}_{K,S}$ whereas the map from $\F_q$-points alone (in \ref{RecClassDiagram}) only witnesses the unramified factors (i.e., those in $K^{\times}_S \backslash \mathbb{I}_K^{(S)}/\mathbb{O}_{K}^{(S)}$). Of course, by \ref{WeilUniform}, these are equivalent; but log geometry has the desirable attribute of allowing us to have a uniform understanding of \textit{all} places in the tamely ramified setting.

\subsection{Extensions} We now describe some potential extensions of the methods of this paper. Most obviously, geometric class field theory is the simplest ($\GL_1$) case of the vast geometric Langlands conjectures \cite{ArinkinGaitsgory2015-SingularSupport}; it is quite plausible that we could apply logarithmic geometry in a similar manner to the above in the context of tamely ramified (global) geometric Langlands for a general reductive group $G$. 

On the other hand, some of the powerful techniques of log geometry have \textit{not} been utilized here; e.g., degenerating the curve $X$. A natural next step would be to consider, for example, logarithmic geometric class field theory for (logarithmic) nodal curves. Additionally, as we have already mentioned, we may be able to extend the geometric underpinning of this work (the ``Div-to-Pic map") using the recently-constructed LogHilb and LogQuot schemes \cite{kennedyhunt}. 

K. Kato has pointed out to the author that another natural extension of this work would be to study the map $\textrm{Ext}^1(\overline{J}, G) \to H^1_{\log} (X(\log S), G)$ for $G$ a \textit{noncommutative} finite group scheme. Finally, we mention that a geometric understanding of higher ``wild" ramification is of great interest \cite{CampbellHayash2021CartierDuality, Guignard2019RamifiedRelativeCurves, Takeuchi2019BlowUps}, and is not explored in this paper. It seems plausible, however, given the geometry of log, that capturing higher ramification in a form analogous to that of the present paper might require some sort of extension of the notion of a logarithmic structure in and of itself. 

\bibliography{refs}{}

@article{Guignard2019RamifiedRelativeCurves,
  author       = {Guignard, Quentin},
  title        = {On the ramified class field theory of relative curves},
  journal      = {Algebra {\&} Number Theory},
  volume       = {13},
  number       = {6},
  pages        = {1299--1326},
  year         = {2019},
  doi          = {10.2140/ant.2019.13.1299},
  eprint       = {1804.02243},
  archivePrefix = {arXiv},
  primaryClass = {math.AG}
}

@article{Takeuchi2019BlowUps,
  author       = {Takeuchi, Daichi},
  title        = {Blow-ups and class field theory for curves},
  journal      = {Algebra {\&} Number Theory},
  volume       = {13},
  number       = {6},
  pages        = {1327--1352},
  year         = {2019},
  doi          = {10.2140/ant.2019.13.1327}
}

@article{ArinkinGaitsgory2015-SingularSupport,
  author       = {Arinkin, Dima and Gaitsgory, Dennis},
  title        = {Singular support of coherent sheaves, and the geometric {L}anglands conjecture},
  journal      = {Selecta Mathematica (New Series)},
  year         = {2015},
  volume       = {21},
  number       = {1},
  pages        = {1--199},
  doi          = {10.1007/s00029-014-0167-5},
  url          = {https://link.springer.com/article/10.1007/s00029-014-0167-5},
  eprint       = {1201.6343},
  archivePrefix= {arXiv},
  primaryClass = {math.AG}
}

@article{BeauvilleLaszlo1995,
  author       = {Arnaud Beauville and Yves Laszlo},
  title        = {Un lemme de descente},
  journal      = {Comptes Rendus de l'Académie des Sciences. Série I. Mathématique},
  volume       = {320},
  number       = {3},
  year         = {1995},
  pages        = {335--340},
  language     = {French},
  issn         = {0764-4442},
  mrclass      = {14D20, 14H60},
  mrnumber     = {1344131}
}

@unpublished{BeilinsonDrinfeldHecke,
  author       = {Beilinson, Alexander and Drinfeld, Vladimir},
  title        = {Quantization of Hitchin's Integrable System and Hecke Eigensheaves},
  note         = {Preprint, available at \url{http://www.math.uchicago.edu/~drinfeld/langlands/}},
  year         = {2000}
}

@misc{CampbellHayash2021CartierDuality,
  author        = {Campbell, Justin and Hayash, Andreas},
  title         = {Geometric class field theory and Cartier duality},
  year          = {2021},
  note          = {arXiv:1710.02892v2},
  eprint        = {1710.02892},
  archivePrefix = {arXiv},
  primaryClass  = {math.AG}
}

@preamble{
   "\def\polhk#1{\setbox0=\hbox{#1}{\ooalign{\hidewidth
    \lower1.5ex\hbox{`}\hidewidth\crcr\unhbox0}}} "
}

@article {FujKat,
    AUTHOR = {Fujiwara, K. and Kato, K.},
     TITLE = {Logarithmic \'etale topology theory},
   JOURNAL = {Unpublished},
      YEAR = {1995},
}

@book{Gabber:Ramero,
    AUTHOR = {Gabber, O. and Ramero},
     TITLE = {Almost Rings and Perfectoid Spaces. Release 8.},
 PUBLISHER = {Online Preprint},
      YEAR = {October 1, 2024},
       URL = {https://pro.univ-lille.fr/fileadmin/user_upload/pages_pros/lorenzo_ramero/hodge.pdf},
}

@book { SGAI,    
AUTHOR = {Grothendieck, A},
     TITLE = {Rev\^etements \'etales et groupe fondamental (SGA 1)},
    SERIES = {Lecture Notes in Mathematics},
    VOLUME = {224},
 PUBLISHER = {Springer-Verlag, Berlin},
      YEAR = {1971},
       URL = {https://arxiv.org/abs/math/0206203},
}

@book{GrothendieckMurre1971,
  author    = {Alexander Grothendieck and Jacob P. Murre},
  title     = {The Tame Fundamental Group of a Formal Neighbourhood of a Divisor with Normal Crossings on a Scheme},
  series    = {Lecture Notes in Mathematics},
  volume    = {208},
  publisher = {Springer-Verlag},
  address   = {Berlin--Heidelberg},
  year      = {1971},
  pages     = {x + 134},
  isbn      = {978-3-540-05499-3},
  doi       = {10.1007/BFb0069608}
}

@article{Hoshi2009LogHomotopy,
  author  = {Hoshi, Yuichiro},
  title   = {The exactness of the log homotopy sequence},
  journal = {Hiroshima Mathematical Journal},
  volume  = {39},
  number  = {1},
  pages   = {61--122},
  year    = {2009},
  month   = mar,
  doi     = {10.32917/hmj/1237392380},
  url     = {https://projecteuclid.org/euclid.hmj/1237392380}
}

@incollection{IllusieLogEt,
     author = {Illusie, Luc},
     title = {An overview of the work of {K.} {Fujiwara,} {K.} {Kato,} and {C.} {Nakayama} on logarithmic \'etale cohomology},
     booktitle = {Cohomologies $p$-adiques et applications arithm\'etiques (II)},
     editor = {Berthelot Pierre and Fontaine Jean-Marc and Illusie Luc and Kato Kazuya and Rapoport Michael},
     series = {Ast\'erisque},
     pages = {271--322},
     publisher = {Soci\'et\'e math\'ematique de France},
     number = {279},
     year = {2002},
     mrnumber = {1922832},
     zbl = {1052.14005},
     language = {en},
     url = {https://www.numdam.org/item/AST_2002__279__271_0/}
}

@article {FukKatShar,
AUTHOR = {Fukaya, Takako and Kato, Kazuya and Sharifi, Romyar},
TITLE = {Toroidal compactifications of the moduli spaces of Drinfeld modules, I},
YEAR = {2020},
JOURNAL = {arXiv preprint arXiv:2008.13376},
MONTH = {08},
PAGES = {},
DOI = {10.48550/arXiv.2008.13376}
}

@misc{kennedyhunt,
      title={The Logarithmic Quot space: foundations and tropicalisation}, 
      author={Patrick Kennedy-Hunt},
      year={2023},
      eprint={2308.14470},
      archivePrefix={arXiv},
      primaryClass={math.AG},
      url={https://arxiv.org/abs/2308.14470}, 
}

@incollection {KatoLogI,
    AUTHOR = {Kato, Kazuya},
     TITLE = {Logarithmic structures of {F}ontaine-{I}llusie},
 BOOKTITLE = {Algebraic analysis, geometry, and number theory ({B}altimore,
              {MD}, 1988)},
     PAGES = {191--224},
 PUBLISHER = {Johns Hopkins Univ. Press, Baltimore, MD},
      YEAR = {1989},
      ISBN = {0-8018-3841-X},
   MRCLASS = {14F30 (14G20)},
  MRNUMBER = {1463703},
MRREVIEWER = {Adolfo\ Quir\'os},
}

@article {KatoLogII,
author = {Kazuya Kato},
title = {Logarithmic Structures of Fontaine-Illusie. II ---Logarithmic Flat Topology},
volume = {44},
journal = {Tokyo Journal of Mathematics},
number = {1},
publisher = {Publication Committee for the Tokyo Journal of Mathematics},
pages = {125 -- 155},
year = {2021},
doi = {10.3836/tjm/1502179316},
URL = {https://doi.org/10.3836/tjm/1502179316}
}

@article{katonakayama1999log,
  title={Log Betti cohomology, log étale cohomology and log de Rham cohomology of log schemes over $\mathbb{C}$},
  author={Kato, Kazuya and Nakayama, Chikara},
  journal={Kodai Mathematical Journal},
  volume={22},
  number={2},
  pages={161--186},
  year={1999},
  publisher={Department of Mathematics, Tokyo Institute of Technology},
  doi={10.2996/kmj/1138044041},
  url={https://projecteuclid.org/euclid.kmj/1138044041}
}

@misc{KatoVal,
      title={Logarithmic Degeneration and Dieudonn\'e Theory (unpublished)}, 
      author={Kazuya Kato},
      year={1992},
      url={https://www.math.brown.edu/dabramov/LOGGEOM/Kato-Dieudonne.pdf}, 
}

@article{MolchoWise2018TheLP,
  title={The logarithmic Picard group and its tropicalization},
  author={Samouil Molcho and Jonathan Wise},
  journal={Compositio Mathematica},
  year={2018},
  volume={158},
  pages={1477 - 1562},
  url={https://api.semanticscholar.org/CorpusID:119661426}
}

@article{MaulikRanganathan_LogDT_2024,
  author  = {Maulik, Davesh and Ranganathan, Dhruv},
  title   = {Logarithmic Donaldson--Thomas theory},
  journal = {Forum of Mathematics, Pi},
  year    = {2024},
  volume  = {12},
  pages   = {1--63},
  eid     = {e9},
  doi     = {10.1017/fmp.2024.1}
}

@book{MilneCFT,
  author    = {Milne, James S.},
  title     = {Class Field Theory},
  year      = {2020},
  edition   = {Version 4.03},
  publisher = {Self‑published, Ann Arbor},
  url       = {https://www.jmilne.org/math/CourseNotes/CFT.pdf},
  note      = {Available online}
}

@article {NakLogEtI,
    AUTHOR = {Nakayama, Chikara},
     TITLE = {Logarithmic \'etale cohomology},
   JOURNAL = {Math. Ann.},
  FJOURNAL = {Mathematische Annalen},
    VOLUME = {308},
      YEAR = {1997},
    NUMBER = {3},
     PAGES = {365--404},
      ISSN = {0025-5831,1432-1807},
   MRCLASS = {14F20},
  MRNUMBER = {1457738},
MRREVIEWER = {Claudio\ Pedrini},
       DOI = {10.1007/s002080050081},
       URL = {https://doi.org/10.1007/s002080050081},
}

@article{NakLogII,
title = {Logarithmic étale cohomology, II},
journal = {Advances in Mathematics},
volume = {314},
pages = {663-725},
year = {2017},
issn = {0001-8708},
doi = {https://doi.org/10.1016/j.aim.2017.05.006},
url = {https://www.sciencedirect.com/science/article/pii/S0001870815303935},
author = {Chikara Nakayama},
}

@mastersthesis{toth2011geometric,
  author       = {Péter Tóth},
  title        = {Geometric Abelian Class Field Theory},
  school       = {Utrecht University},
  year         = {2011},
  type         = {Master's thesis},
  url          = {https://studenttheses.uu.nl/handle/20.500.12932/7282}
}
\bibliographystyle{alpha}

\end{document}